\documentclass[11pt,twoside]{article}
\usepackage{times}
\usepackage{amsmath,amssymb,amsthm}
\usepackage{enumerate}
\usepackage{cite}
\usepackage{graphicx}
\usepackage{booktabs}
\usepackage{makecell}
\usepackage{epstopdf}
\usepackage{xcolor}

\allowdisplaybreaks

\newcommand{\supp}{\operatorname{supp}}
\newcommand{\R}{\mathbb R}

\newcommand{\p}{\partial}
\newcommand{\ve}{\varepsilon}
\newcommand{\f}{\frac}

\newcommand{\al}{\alpha}

\newcommand{\Acal}{\mathcal A}
\newcommand{\Bcal}{\mathcal B}
\newcommand{\dd}{\,\mathrm{d}}
\newcommand{\Ecal}{\mathcal E}
\newcommand{\Tcal}{\mathcal T}
\newcommand{\Rcal}{\mathcal R}

\allowdisplaybreaks

\theoremstyle{plain}
\newtheorem{theorem}{Theorem}[section]
\newtheorem{proposition}{Proposition}[section]
\newtheorem{lemma}[theorem]{Lemma}

\theoremstyle{definition}

\theoremstyle{remark}
\newtheorem{remark}{Remark}[section]

\numberwithin{equation}{section}

\title{Global small data radial symmetric solutions of 3D semilinear Euler-Poisson-Darboux equations}

\author{Li Qianqian$^{1}$,\quad Xin Qiao$^{2}$,\quad Yin Huicheng$^{3,}$\footnote{Li Qianqian (\texttt{214597007@qq.com}) and Yin Huicheng
    (\texttt{huicheng@} \texttt{nju.edu.cn}, \texttt{05407@njnu.edu.cn}) are supported by the National Natural Science Foundation of China
    (No.~12331007, No.~12571237).  Xin Qiao (\texttt{xinqiaoylsy@163.com}) is supported by the National Natural Science Foundation of China
    (No.~12261092) and the Scientific
Research Program of the Higher Education Institution of XinJiang (No. XJEDU2025P090).}\vspace{0.5cm}\\
\small
1. College of Mathematics and System Science, Xinjiang University, Urumqi 830017, China.\\
\small
2. School of Mathematics and Statistics, Yili Normal University, Yining, 835000, China.\\
\small
3. School of Mathematical Sciences and IMS, Nanjing Normal University, Nanjing 210023, China.\\
}
\vspace{0.5cm}

\begin{document}

\date{}

\maketitle
\thispagestyle{empty}

\begin{abstract}
For the 3D semilinear Euler-Poisson-Darboux equation
$\square u+\f{\mu}{t}\p_tu=|u|^p$, where $t\geq1$, $\mu>0$ and $p>1$,
it is conjectured that there is a critical exponent  $p_{crit}(3,\mu)=\max\{p_s(3+\mu), p_f(3)\}$
with the Strauss exponent
$p_s(3+\mu)=\f{\mu+4+\sqrt{\mu^2+16\mu+32}}{2(\mu+2)}$
and the Fujita exponent $p_f(3)=\f{5}{3}$ such that when $p>p_{crit}(3,\mu)$, the small data solution $u$ exists globally,
otherwise, when $1<p\le p_{crit}(3,\mu)$, the solution $u$ can blow up in finite time.
It is pointed out that for $1<p\le p_{crit}(3,\mu)$, the blowup of solution $u$ has been shown.
However, it is still open for the global existence of small solution $u$ when $p>p_{crit}(3,\mu)$.
Note that $p_{crit}(3,\mu)=p_s(3+\mu)$ for $0<\mu<\f{14}{5}$ and $p_{crit}(3,\mu)=p_f(3)$ for
$\mu\geq\f{14}{5}$.
In the recent paper \cite{LY-2}, the authors have obtained the
global small solution $u$ for $p>\max\{\f{5}{3},1+\f{2}{\mu}\}$ and $\mu\geq\f{14}{5}$. In this paper, by
utilizing the hypergeometric Riemann representation and establishing some delicate pointwise spacetime weighted estimates, we prove
the global existence of small data radial solution $u$ in the remaining
range of $\f{5}{3}<p\leq1+\f{2}{\mu}$ and $\f{14}{5}\leq\mu<3$. Therefore, for
the radially symmetric case and $\mu\geq\f{14}{5}$, the global existence problem of small data solution
$u$ is solved when $p>p_{crit}(3,\mu)=\f{5}{3}$.
\end{abstract}


\qquad \quad hypergeometric Riemann representation

\vskip 0.1 true cm


\tableofcontents

\section{Introduction}\label{Sec1}
In this paper, we are concerned with the 3D semilinear Euler-Poisson-Darboux equation
\begin{equation}
\label{qq:1}
\left\{
\begin{aligned}
&\square u+\f{\mu}{t}\p_tu=|u|^p,
&&t\geq1,&x\in\Bbb R^3,\\
&(u,\p_tu)(1,x)=\ve(u_0,u_1)(x),
\end{aligned}
\right.
\end{equation}
where $\mu>0$, $p>1$, $\square=\p_t^2-\Delta$,
$\Delta=\p_1^2+\p_2^2+\p_3^2$, $\p_j=\p_{x_j}$ ($j=1,2,3$),
$(u_0,u_1)\in C_0^\infty(\Bbb R^3)\times C_0^\infty(\Bbb R^3)$ with
$\operatorname{supp}(u_0,u_1)\subset B(0,1)$, and $\ve>0$ is sufficiently small.
For the detailed introductions on
the physical background of the linear operator
$\square+\f{\mu}{t}\p_t$, one can see \cite{LWY} and the references therein.

Denote by $p_{crit}(n,\mu)=\max\{p_s(n+\mu),p_f(n)\}$ for $n\in\Bbb N$ and $n\ge 2$, where
$p_s(z)=\f{z+1+\sqrt{z^2+10z-7}}{2(z-1)}$ $(z>1)$ is the Strauss
exponent which corresponds to the positive root of the quadratic equation
$(z-1)p^2-(z+1)p-2=0$, and $p_f(n)=1+\f{2}{n}$ is the Fujita exponent.
The Strauss exponent
$p_s(n)=\f{n+1+\sqrt{n^2+10n-7}}{2(n-1)}$ $(n\geq2)$ arises from the critical
exponent problem for the global existence and blowup of
solutions to the $n$-dimensional  semilinear wave equation
$\square v=|v|^p$ (see \cite{Strauss}). The Fujita exponent
$p_f(n)=1+\f{2}{n}$ arises from the corresponding $n$-dimensional
semilinear parabolic equation
$\partial_tv-\Delta v=|v|^p$ (see \cite{Fuj}).
Set $p_{crit}(n,\mu)=\max\{p_s(n+\mu), p_f(n)\}$. Then direct computation
yields $p_{crit}(3,\mu)=p_s(3+\mu)$ for $0<\mu\le\f{14}{5}$ and
$p_{crit}(3,\mu)=p_f(3)=\f{5}{3}$ for $\mu\geq\f{14}{5}$.
As proposed in \cite{Rei1,Imai} and \cite{LY-2}, there is an open question as follows:

\vskip 0.1 true cm
\noindent
{\bf Open question (A).}
{\it For the 3D problem \eqref{qq:1},

\noindent
{\bf (A1)} if $\mu\geq\f{14}{5}$ and
$p>p_{crit}(3,\mu)=\f{5}{3}$, the small solution $u$ exists globally;

\noindent
{\bf (A2)} if $0<\mu<\f{14}{5}$ and
$p>p_{crit}(3,\mu)=p_s(3+\mu)$, there exists a global small data solution $u$.
}

\medskip

Note that the finite time blowup result for the corresponding $n-$dimensional problem \eqref{qq:1} has been established for
$1<p\leq p_{crit}(n,\mu)=\max\{p_s(n+\mu),p_f(n)\}$, see
\cite{FLT,IS,LTW,PR-0,PR,TL2,W1}. On the other hand,
in recent paper \cite{LY-2}, which also contains a detailed summary on the known global existence results for
$p>p_{crit}(n,\mu)$ and the remaining open cases (see also \cite{DA-0}-\cite{Rei1}, \cite{HL}-\cite{HLY}, \cite{LZ-1}-\cite{LY}
and \cite{P1}), the authors utilize the vector field method to prove that  the small solution of  \eqref{qq:1} exists globally
when
\begin{equation}\label{YHC-01}
\begin{aligned}
\text{$p>\max\big\{\f{5}{3},1+\f{2}{\mu}\big\}$ and $\mu\geq\f{14}{5}$.}
\end{aligned}
\end{equation}

Due to
\begin{equation}\label{YHC-02}
\begin{aligned}
\text{$\max\big\{\f{5}{3},1+\f{2}{\mu}\big\}=\f{5}{3}$ for $\mu\geq3$},
\end{aligned}
\end{equation}
it follows from Theorem 1.1 of \cite{LY-2} and \eqref{YHC-01}
that the {\bf Open question (A1)} has already been solved when $\mu\geq3$. For
$\f{14}{5}\leq\mu<3$, however, the previous global existence result in Theorem 1.1 of \cite{LY-2} requires
the restriction of $p>1+\f{2}{\mu}$, which leaves the scope
\begin{equation}\label{YHC-03}
\begin{aligned}
\text{$\f{5}{3}<p\leq1+\f{2}{\mu}$ with $\f{14}{5}\leq\mu<3$}
\end{aligned}
\end{equation}
such that {\bf Open question (A1)} is still open.
In the present paper, by use of the different method from
that in \cite{LY-2}, we solve problem  \eqref{qq:1}
for the radial symmetric solution $u$ under the condition \eqref{YHC-03}.
The main result can be  stated as
\begin{theorem}
\label{thm:main}
For $\f{5}{3}<p\leq1+\f{2}{\mu}$ and
$\f{14}{5}\leq\mu<3$, it is assumed that $(u_0,u_1)\in C_0^\infty(\Bbb R^3)\times C_0^\infty(\Bbb R^3)$ are radially symmetric. Then there exists a
constant $\ve_0>0$ such that when $0<\ve\leq\ve_0$, problem
\eqref{qq:1} admits a unique global  radial solution
$u\in C\bigl([1,\infty)\times\Bbb R^3\bigr)\cap
C^2\bigl([1,\infty)\times(\Bbb R^3\setminus\{0\})\bigr)$.
\end{theorem}

\begin{remark}\label{re1}
Combining the global existence result in Theorem 1.1 of \cite{LY-2} for
$p>\max\{\f{5}{3},1+\f{2}{\mu}\}$ and $\mu\geq\f{14}{5}$
in the present paper,
{\bf Open question (A1)} is completely solved for the radially symmetric case.
\end{remark}

\begin{remark}\label{re-2}
For problem \eqref{qq:1}, the global existence of small data solution $u$
has been known for $\mu\in(0,1)\cup(1,2)$ and $p>p_{crit}(3,\mu)$, see
\cite[Remark~1.14]{HLY}. Thus,
{\bf Open question (A2)} remains open  for $\mu=1$ or
$\mu\in[2,\f{14}{5})$ and $p>p_{crit}(3,\mu)$.
\end{remark}

Next, we give the comments for the proof of Theorem~\ref{thm:main}.
Motivated by the weighted physical-space method in \cite{Asakura1986,Rei1}, we introduce
the following Liouville transformation to rewrite the equation in problem \eqref{qq:1} for the radial solution $u(t,x)=u(t,r)$ $(r=|x|
=\sqrt{x_1^2+x_2^2+x_3^2})$:
\begin{equation}\label{YHC-04}
\begin{aligned}
\text{$\psi(t,r)=rw(t,r)$ with $ w(t,r)=t^{\f{\mu}{2}}u(t,r)$.}
\end{aligned}
\end{equation}
Extending $w$ evenly to $r\in\R$, one has that
\begin{equation}
\label{qq:method-reduction}
\left\{ \enspace
\begin{aligned}
&\p_t^2\psi-\p_r^2\psi-\f{\lambda(\lambda+1)}{t^2}\psi=G_w(t,r)=rt^{-\beta}|w(t,|r|)|^p,\\
&\psi|_{t=1}=\ve\psi_0(r)=\ve ru_0(r), \quad \p_t\psi|_{t=1}=\ve\psi_1(r)=\ve r\big(u_1(r)+\f{\mu}{2}u_0(r)\big),
\end{aligned}
\right.
\end{equation}
where
\begin{equation}\label{qq:method-reduction'}
\lambda=\f{\mu-2}{2},\quad \beta=\f{\mu}{2}(p-1).
\end{equation}
For $\mu=2$ and then $\lambda=0$, the homogeneous part of
\eqref{qq:method-reduction} becomes the 1D wave equation
which has been treated in \cite{Rei1}. For
$\f{14}{5}\leq\mu<3$, that is,
$\f{2}{5}\leq\lambda<\f{1}{2}$, the extra troublesome term
$-\lambda(\lambda+1)t^{-2}\psi$ appears in \eqref{qq:method-reduction}.  To investigate
\eqref{qq:method-reduction}, we first give the direct expression for the solution of \eqref{qq:method-reduction}
by making full use of the Riemann representation formulae whose
kernel is  expressed in terms of the Gauss hypergeometric function.
Based on this, some suitably spacetime weighted pointwise estimates can be established for \eqref{qq:method-reduction}.
The main ingredients are summarized as follows.

\medskip

${\bf \bullet}$ Set
\begin{equation}
\label{qq:method-kernel}
\Ecal_\lambda(t,s,\sigma)
=(ts)^{-\lambda}\bigl((t+s)^2-\sigma^2\bigr)^\lambda
{}_2F_1\big(-\lambda,-\lambda;1;
\f{(t-s)^2-\sigma^2}{(t+s)^2-\sigma^2}\big),
\end{equation}
where the Gauss hypergeometric function ${}_2F_1(\alpha,\gamma;\delta;z)$ admits the following expression for $|z|<1$,
\[
{}_2F_1(\alpha,\gamma;\delta;z)
=\sum_{k=0}^{\infty}\f{(\alpha)_k(\gamma)_k}{(\delta)_k k!}z^k
\]
with $(a)_k=a(a+1)\cdots(a+k-1)$ being the Pochhammer symbol. One can see more properties for  the
hypergeometric functions in \cite[Section~15]{OLBC}. By
\eqref{qq:method-kernel} and \eqref{qq:method-reduction}, one formally has (see Section~\ref{sec:preliminaries} below)
\begin{equation}
\label{Q2-01}
\begin{aligned}
w(t,r)=w_{\rm lin}(t,r)+(\Tcal w)(t,r),
\end{aligned}
\end{equation}
where
\begin{equation*}
\label{Q2-001}
\begin{aligned}
w_{\rm lin}(t,r)&={}\f{\ve}{2r}
\bigl[\psi_0(r+t-1)+\psi_0(r-t+1)\bigr]
+\f{c_\lambda\ve}{r}
\int_{r-t+1}^{r+t-1}
K_0(t,r-\rho)\psi_0(\rho)\mathrm d\rho
\\
&+\f{c_\lambda\ve}{r}
\int_{r-t+1}^{r+t-1}
K_1(t,r-\rho)\psi_1(\rho)\,\mathrm d\rho,\\
(\Tcal w)(t,r)&=\f{c_\lambda}{r}
\int_1^t\int_{r-(t-s)}^{r+(t-s)}
\Ecal_\lambda(t,s,r-\rho)G_w(s,\rho)\mathrm d\rho\mathrm ds\\
\end{aligned}
\end{equation*}
with $K_1(t,\sigma)=\Ecal_\lambda(t,1,\sigma)$, $K_0(t,\sigma)=
-\left.\p_s\Ecal_\lambda(t,s,\sigma)\right|_{s=1}$ and $c_\lambda=2^{-(2\lambda+1)}$.

Therefore, we naturally define the mapping from \eqref{Q2-01}
\begin{equation}
\label{Q2-02}
\begin{aligned}
\Phi(w)(t,r)={}&w_{\rm lin}(t,r)+(\Tcal w)(t,r).
\end{aligned}
\end{equation}
If the mapping \eqref{Q2-02} is shown to have a unique fixed point in such a weighted radial function space
$X_\kappa$ equipped with the norm
\begin{equation}
\label{qq:method-space}
\|w\|_{X_\kappa}
=
\sup_{\substack{t\geq1\\0\leq r\leq t}}
\big\{
(1+t+r)(1+t-r)^{\kappa-1}|w(t,r)|
+
\f{(1+t+r)(1+t-r)^{\kappa-1}}{1+r}
|\p_r(rw)(t,r)|
\big\}
\end{equation}
with $\f{2-\beta}{p-1}
<\kappa<\min\{p+\beta-1,\,2-\lambda\}$, $\lambda=\f{\mu-2}{2}$ and $\beta=\f{\mu}{2}(p-1)$,
then Theorem \ref{thm:main} can be proved (see details in Sections~\ref{sec:linear-estimates}-
\ref{Sect5}).

\medskip

${\bf \bullet}$ Note that
$G_w(s,\rho)=\rho s^{-\beta}|w(s,|\rho|)|^p$
is odd in $\rho$, one then has
\begin{equation}
\label{qq:method-pairing}
\begin{aligned}
c_\lambda^{-1}r(\Tcal w)(t,r)
={}&\int_1^t\int_{|t-s-r|}^{t-s+r}
\Ecal_\lambda(t,s,r-q)G_w(s,q)\,\mathrm dq\,\mathrm ds\\
&+\int_1^{t-r}\int_0^{t-s-r}
\bigl[\Ecal_\lambda(t,s,r-q)-\Ecal_\lambda(t,s,r+q)\bigr]
G_w(s,q)\,\mathrm dq\,\mathrm ds.
\end{aligned}
\end{equation}
It is emphasized that the kernel difference
$\Ecal_\lambda(t,s,r-q)-\Ecal_\lambda(t,s,r+q)$
in the second integral of \eqref{qq:method-pairing} will play essential roles in
deriving the contractivity of $(\Tcal w)$ for
$\f{5}{3}<p\leq1+\f{2}{\mu}$. The corresponding estimates are given in
Section~\ref{sec:nonlinear-Riemann-operator}.

\medskip

Based on the conclusions above, by returning to the transformation
$u(t,r)=t^{-\mu/2}w(t,r)$, then the 3D problem \eqref{qq:1} admits a global small data radial
symmetric solution $u\in C\bigl([1,\infty)\times\Bbb R^3\bigr)
\cap C^2\bigl([1,\infty)\times(\Bbb R^3\setminus\{0\})\bigr)$.

This paper is organized as follows. In Section~2, at first, we apply the
Liouville transformation to reformulate the original
3D problem \eqref{qq:1} with the radial symmetric initial data into a 1D semilinear wave equation including a
time-dependent zeroth-order term. Secondly, we give the representation formulae
for the solutions to the corresponding homogeneous and inhomogeneous linear problems.
In addition, some basic properties of
the Gauss hypergeometric functions are also recalled.
In Sections~3-4, some weighted spacetime pointwise estimates for the
homogeneous and  inhomogeneous problems are established, respectively. Based on the estimates
obtained in Sections~3--4 and the contraction mapping principle,  the proof of Theorem~\ref{thm:main}
is completed in  Section \ref{Sect5}.

\vskip 0.3 true cm

\noindent\textbf{Notations}

\vskip 0.1 true cm

$\bullet$ For nonnegative quantities $f$ and $g$, $f\lesssim g$ means
$f\leq Cg$ for some generic constant $C>0$ independent of $\ve$ and
the spacetime variables, and $f\sim g$ means
$g\lesssim f\lesssim g$.

\vskip 0.1 true cm

$\bullet$
For $y\in\R$, set $\langle y\rangle=(1+|y|^2)^{1/2}$.

\vskip 0.1 true cm

$\bullet$ For $\f{5}{3}<p\leq1+\f{2}{\mu}$ and $\f{14}{5}\leq\mu<3$, denote
\begin{equation*}\label{nota04}
a=\f{\mu}{2},\quad
\lambda=a-1=\f{\mu-2}{2}\in[\f{2}{5}, \f{1}{2}),\quad
\beta=a(p-1)=\f{\mu}{2}(p-1)\in(0, 1].
\end{equation*}

$\bullet$  For $t\geq s\geq1$ and $|\sigma|\leq t-s$, define
\begin{equation*}\label{nota01}
D(t,s,\sigma)=(t+s)^2-\sigma^2,\quad
Q(t,s,\sigma)=\f{D(t,s,\sigma)}{4ts}\geq1
\end{equation*}
and
\begin{equation*}\label{nota02}
0\leq z(t,s,\sigma)=
\f{(t-s)^2-\sigma^2}{(t+s)^2-\sigma^2}
=1-Q(t,s,\sigma)^{-1}<1.
\end{equation*}

\section{Preliminaries}\label{sec:preliminaries}

In this section, we first reformulate problem \eqref{qq:1} with the radial symmetric initial data into a 1D
semilinear wave equation with a time-dependent zeroth-order term. Subsequently, the representation formulae for the
solutions to the corresponding homogeneous
and inhomogeneous linear problems are given.

\subsection{Liouville transformation and reformulation}
\label{liou}
For the radial symmetric solution $u$ of problem \eqref{qq:1}, we write $u(t,x)=u(t,r)$ with $r=|x|$.
By introducing the Liouville transformation
$w(t,r)=t^{\f{\mu}{2}}u(t,r)$,
then it follows from \eqref{qq:1} that
\begin{equation}
\label{eq:w-equation}
\left\{
\begin{aligned}
&\p_t^2w-\p_r^2w-\f{2}{r}\p_r w
-\f{\lambda(\lambda+1)}{t^2}w
=t^{-\beta}|w|^p,
&&t\geq1,\quad r\geq0,\\
&(w,\p_tw)(1,r)
=\ve\bigl(u_0,u_1+a u_0\bigr)(r),
\end{aligned}
\right.
\end{equation}
where $a$, $\lambda$, and $\beta$ have been defined in Section \ref{Sec1}. We extend $w(t,\cdot)$ evenly to $\R$ and set
\begin{equation}
\label{eq:psi}
\psi(t,r)=rw(t,r),\qquad r\in\R.
\end{equation}
Then it follows from \eqref{eq:w-equation} that
\begin{equation}
\label{eq:psi-equation}
\left\{
\begin{aligned}
&\p_t^2\psi-\p_r^2\psi
-\f{\lambda(\lambda+1)}{t^2}\psi
=G_w(t,r),
&&t\geq1,\quad r\in\R,\\
&(\psi,\p_t\psi)(1,r)
=\ve(\psi_0,\psi_1)(r)
=\ve r\bigl(u_0,u_1+a u_0\bigr)(r),
\end{aligned}
\right.
\end{equation}
where
\begin{equation}
\label{eq:Gw}
G_w(t,r)=rt^{-\beta}|w(t,|r|)|^p.
\end{equation}
Note that $\psi$ and $G_w$ are odd in $r$, and in particular,
$\psi(t,0)=0$.

\subsection{The hypergeometric kernel}

We now recall some basic properties on
the Gauss hypergeometric function, which will be used later.

\begin{lemma}[Lemma 1 of \cite{P2}]
\label{lem:hypergeometric}
Let $\alpha,\gamma\in\R$ and
$\delta\in\R\setminus\{0,-1,-2,\ldots\}$. Then the Gauss
hypergeometric function ${}_2F_1(\alpha,\gamma;\delta;z)$ satisfies
the following properties:
\begin{enumerate}[(i)]
\item ${}_2F_1(\alpha,\gamma;\delta;z)$ is a solution to the  hypergeometric
differential equation
\begin{equation} \label{eq:hyp-ode} z(1-z)H''(z) +\bigl[\delta-(\alpha+\gamma+1)z\bigr]H'(z) -\alpha\gamma H(z)=0; \end{equation}

\item
\[
{}_2F_1(\alpha,\gamma;\delta;0)=1.
\]
\end{enumerate}
\end{lemma}

Before deriving the representation formulae for the corresponding homogeneous and
inhomogeneous linear problems of \eqref{eq:psi-equation}, we next establish the following
properties of the kernel $\Ecal_\lambda$ defined in \eqref{qq:method-kernel}.

\begin{lemma}
\label{lem:kernel-equation}
For $t>s\geq1$ and $|\sigma|<t-s$, it holds that
\begin{equation}
\label{eq:kernel-t}
\big(
\p_t^2-\p_\sigma^2-\f{\lambda(\lambda+1)}{t^2}
\big)
\Ecal_\lambda(t,s,\sigma)=0
\end{equation}
and
\begin{equation}
\label{eq:kernel-s}
\big(
\p_s^2-\p_\sigma^2-\f{\lambda(\lambda+1)}{s^2}
\big)
\Ecal_\lambda(t,s,\sigma)=0.
\end{equation}
Moreover, for $|\sigma|=t-s$, one has
\begin{equation}
\label{eq:kernel-characteristic}
\Ecal_\lambda(t,s,\pm(t-s))=4^\lambda,
\quad
2c_\lambda\Ecal_\lambda(t,s,\pm(t-s))=1.
\end{equation}
\end{lemma}

\begin{proof}
Set
\[
H(z)={}_2F_1(-\lambda,-\lambda;1;z),
\quad
\mathcal F_\lambda(Q)=Q^\lambda H(1-Q^{-1}).
\]
Then
\[
\Ecal_\lambda(t,s,\sigma)
=4^\lambda\mathcal F_\lambda\bigl(Q(t,s,\sigma)\bigr).
\]
We now prove \eqref{eq:kernel-t}. Fixing $s$ and differentiating
$Q=Q(t,s,\sigma)$ yield
\[
\p_tQ=\f{t^2-s^2+\sigma^2}{4st^2},
\quad
\p_\sigma Q=-\f{\sigma}{2ts}.
\]
Then it follows from a direct computation that
\begin{equation}
\label{eq:Q-identities-t}
(\p_t Q)^2-(\p_\sigma Q)^2
=\f{Q(Q-1)}{t^2},
\quad
\p_t^2Q-\p_\sigma^2 Q
=\f{2Q-1}{t^2}.
\end{equation}
In addition, one has
\begin{align}
(\p_t^2-\p_\sigma^2)\Ecal_\lambda
&=
4^\lambda
\left[
\mathcal F_\lambda''(Q)
\big((\p_t Q)^2-(\p_\sigma Q)^2\big)
+\mathcal F_\lambda'(Q)
\big(\p_t^2Q-\p_\sigma^2 Q\big)
\right]\notag\\
&=
\f{4^\lambda}{t^2}
\left[
Q(Q-1)\mathcal F_\lambda''(Q)
+(2Q-1)\mathcal F_\lambda'(Q)
\right].
\label{eq:E-chain}
\end{align}
Let $z=1-Q^{-1}$. Due to
$\mathcal F_\lambda(Q)=Q^\lambda H(z)$, a direct calculation gives
\begin{align}
&Q(Q-1)\mathcal F_\lambda''(Q)
+(2Q-1)\mathcal F_\lambda'(Q)
-\lambda(\lambda+1)\mathcal F_\lambda(Q)\notag\\
&\qquad
=
Q^\lambda(1-z)
\left[
z(1-z)H''(z)
+\bigl(1-(1-2\lambda)z\bigr)H'(z)
-\lambda^2H(z)
\right].
\label{eq:F-hyper}
\end{align}
By Lemma~\ref{lem:hypergeometric} (i) with
$\alpha=\gamma=-\lambda$ and $\delta=1$, we arrive at
\[
z(1-z)H''(z)
+\bigl[1-(1-2\lambda)z\bigr]H'(z)
-\lambda^2H(z)=0.
\]
Hence,
\[
Q(Q-1)\mathcal F_\lambda''(Q)
+(2Q-1)\mathcal F_\lambda'(Q)
=
\lambda(\lambda+1)\mathcal F_\lambda(Q).
\]
Substituting this into \eqref{eq:E-chain} yields
\[
(\p_t^2-\p_\sigma^2)\Ecal_\lambda(t,s,\sigma)
=
\f{\lambda(\lambda+1)}{t^2}
\Ecal_\lambda(t,s,\sigma),
\]
which proves \eqref{eq:kernel-t}.
From this, the symmetry of $\Ecal_\lambda(t,s,\sigma)$ with respect to $t$ and $s$
allows us to obtain \eqref{eq:kernel-s}.

On the other hand, for $|\sigma|=t-s$,  it follows from a direct computation that
\[
D(t,s,\pm(t-s))=4ts,\quad Q(t,s,\pm(t-s))=1,\quad z(t,s,\pm(t-s))=0.
\]
This, together with Lemma~\ref{lem:hypergeometric} (ii), yields
\begin{equation}\label{QQ-36''}
\Ecal_\lambda(t,s,\pm(t-s))=4^\lambda,
\quad
2c_\lambda\Ecal_\lambda(t,s,\pm(t-s))=1.
\end{equation}
\end{proof}

\subsection{Riemann representation formulae}

\begin{lemma}[{\bf Homogeneous problem with nontrivial initial data}]
\label{lem:homogeneous-representation}
For $\psi_0$ and $\psi_1$ defined in \eqref{eq:psi-equation},  the solution of
\begin{equation}
\label{eq:hom-linear}
\left\{
\begin{aligned}
&\p_t^2\psi-\p_r^2\psi-\f{\lambda(\lambda+1)}{t^2}\psi=0,
&&t\geq1,\quad r\in\R,\\
&\psi(1,r)=\ve\psi_0(r),\quad
\p_t\psi(1,r)=\ve\psi_1(r),
\end{aligned}
\right.
\end{equation}
is given by
\begin{equation}
\label{eq:hom-representation}
\begin{aligned}
\psi(t,r)={}&\f{\ve}{2}\psi_0(r+t-1)+\f{\ve}{2}\psi_0(r-t+1)
+\ve c_\lambda\int_{r-t+1}^{r+t-1}K_0(t,r-\rho)\psi_0(\rho)\dd\rho\\
&+\ve c_\lambda\int_{r-t+1}^{r+t-1}K_1(t,r-\rho)\psi_1(\rho)\dd\rho,
\end{aligned}
\end{equation}
where
\begin{equation}
\label{eq:K01}
K_1(t,\sigma)=\Ecal_\lambda(t,1,\sigma),\quad
K_0(t,\sigma)=-\left.\p_s\Ecal_\lambda(t,s,\sigma)\right|_{s=1},
\quad
c_\lambda=2^{-(2\lambda+1)}.
\end{equation}
\end{lemma}

\begin{proof}
For $t>1$ and $r\in\R$, set
\begin{equation}\label{QQ-5'}
\Omega_{t,r}
=
\bigl\{(s,\rho):1<s<t,\
r-(t-s)<\rho<r+(t-s)\bigr\}
\end{equation}
and define
\begin{equation}\label{QQ-5}
\Rcal_\lambda(t,r;s,\rho)
=
2c_\lambda\Ecal_\lambda(t,s,r-\rho).
\end{equation}
For $(s,\rho)\in\Omega_{t,r}$, let $\sigma=r-\rho$. Then
\[
|\sigma|<t-s,
\quad
\p_\rho^2\Ecal_\lambda(t,s,r-\rho)
=
\p_\sigma^2\Ecal_\lambda(t,s,\sigma).
\]
Hence, it follows from Lemma~\ref{lem:kernel-equation} that
\begin{equation}\label{QQ-4}
\begin{aligned}
\big(
\p_s^2-\p_\rho^2-\f{\lambda(\lambda+1)}{s^2}
\big)
\Rcal_\lambda(t,r;s,\rho)
=0.
\end{aligned}
\end{equation}
Moreover, by \eqref{eq:kernel-characteristic}, one has
\begin{equation}\label{QQ-1}
\Rcal_\lambda=1
\quad\text{for}\quad \rho=r\pm(t-s).
\end{equation}
Then
\begin{equation}\label{QQ-2}
\begin{aligned}
(\p_s-\p_\rho)\Rcal_\lambda&=0
&&\text{for}\quad  \rho=r+(t-s),\\
(\p_s+\p_\rho)\Rcal_\lambda&=0
&&\text{for}\quad  \rho=r-(t-s).
\end{aligned}
\end{equation}
Multiplying \eqref{eq:hom-linear} by $\Rcal_\lambda$ and \eqref{QQ-4} by $\psi$, it follows that for $(s,\rho)\in\Omega_{t,r}$,
\begin{equation}\label{QQ-3}
\p_s\left(
\Rcal_\lambda\p_s\psi-\psi\p_s\Rcal_\lambda
\right)
-
\p_\rho\left(
\Rcal_\lambda\p_\rho\psi-\psi\p_\rho\Rcal_\lambda
\right)
=0.
\end{equation}
Set
\[
A(s,\rho)=\Rcal_\lambda\p_s\psi-\psi\p_s\Rcal_\lambda,
\quad
B(s,\rho)=\Rcal_\lambda\p_\rho\psi-\psi\p_\rho\Rcal_\lambda.
\]
Integrating \eqref{QQ-3} over $\Omega_{t,r}$ and applying
\[
\f{\dd}{\dd s}\int_{r-t+s}^{r+t-s}A(s,\rho)\dd\rho
=
\int_{r-t+s}^{r+t-s}\p_sA(s,\rho)\dd\rho
-A(s,r+t-s)-A(s,r-t+s)
\]
and
\[
\int_{r-t+s}^{r+t-s}\p_\rho B(s,\rho)\dd\rho
=
B(s,r+t-s)-B(s,r-t+s),
\]
we arrive at
\begin{align}
0={}&
-\int_{r-t+1}^{r+t-1}
\left[
\Rcal_\lambda\p_s\psi-\psi\p_s\Rcal_\lambda
\right]_{s=1}\dd\rho\notag\\
&+\int_1^t
\left[
\Rcal_\lambda(\p_s\psi-\p_\rho\psi)
-\psi(\p_s-\p_\rho)\Rcal_\lambda
\right]_{\rho=r+(t-s)}\dd s\notag\\
&+\int_1^t
\left[
\Rcal_\lambda(\p_s\psi+\p_\rho\psi)
-\psi(\p_s+\p_\rho)\Rcal_\lambda
\right]_{\rho=r-(t-s)}\dd s.
\label{eq:Riemann-boundary}
\end{align}
In addition, \eqref{QQ-1} and \eqref{QQ-2} imply
\begin{equation}\label{QQ-6}
\begin{aligned}
\int_1^t
\left[
\Rcal_\lambda(\p_s\psi-\p_\rho\psi)
-\psi(\p_s-\p_\rho)\Rcal_\lambda
\right]_{\rho=r+(t-s)}\dd s&
=\int_1^t(\p_s\psi-\p_\rho\psi)(s,r+t-s)\dd s\\
&=\int_1^t\f{\dd}{\dd s}\psi(s,r+t-s)\dd s\\
&=\psi(t,r)-\ve\psi_0(r+t-1).
\end{aligned}
\end{equation}
Similarly, it holds that
\begin{equation}\label{QQ-7}
\begin{aligned}
&\int_1^t
\left[
\Rcal_\lambda(\p_s\psi+\p_\rho\psi)
-\psi(\p_s+\p_\rho)\Rcal_\lambda
\right]_{\rho=r-(t-s)}\dd s
=\psi(t,r)-\ve\psi_0(r-t+1).
\end{aligned}
\end{equation}
By \eqref{eq:hom-linear}, \eqref{eq:K01} and \eqref{QQ-5}, we have
\begin{equation}\label{QQ-8}
\begin{aligned}
&\int_{r-t+1}^{r+t-1}
\left[
\Rcal_\lambda\p_s\psi-\psi\p_s\Rcal_\lambda
\right]_{s=1}\dd\rho=\ve
\int_{r-t+1}^{r+t-1}
\left[
\Rcal_\lambda(t,r;1,\rho)\psi_1(\rho)
-\psi_0(\rho)\p_s\Rcal_\lambda(t,r;1,\rho)
\right]\dd\rho\\
&=2\ve c_\lambda\int_{r-t+1}^{r+t-1}
K_1(t,r-\rho)\psi_1(\rho)\dd\rho
+2\ve c_\lambda\int_{r-t+1}^{r+t-1}
K_0(t,r-\rho)\psi_0(\rho)\dd\rho.
\end{aligned}
\end{equation}
Substituting \eqref{QQ-6}-\eqref{QQ-8} into \eqref{eq:Riemann-boundary} yields
\begin{align}
2\psi(t,r)
={}&\ve\psi_0(r+t-1)+\ve\psi_0(r-t+1)
+2\ve c_\lambda\int_{r-t+1}^{r+t-1}
K_1(t,r-\rho)\psi_1(\rho)\dd\rho\notag\\
&+2\ve c_\lambda\int_{r-t+1}^{r+t-1}
K_0(t,r-\rho)\psi_0(\rho)\dd\rho.
\label{eq:Riemann-hom}
\end{align}
Therefore, \eqref{eq:hom-representation} is shown.
\end{proof}
\begin{lemma}[{\bf Inhomogeneous problem with vanishing data}]
\label{lem:inhomogeneous-representation}
For $G\in C^1([1,\infty)\times\R)$, the solution of
\begin{equation}
\label{eq:inhom-linear}
\left\{
\begin{aligned}
&\p_t^2\psi-\p_r^2\psi-\f{\lambda(\lambda+1)}{t^2}\psi=G,
&&t\geq1,\quad r\in\R,\\
&\psi(1,r)=0,\quad \p_t\psi(1,r)=0
\end{aligned}
\right.
\end{equation}
is given by
\begin{equation}
\label{eq:inhom-representation}
\psi(t,r)
=c_\lambda\int_1^t
\int_{r-(t-s)}^{r+(t-s)}
\Ecal_\lambda(t,s,r-\rho)G(s,\rho)\dd\rho\dd s.
\end{equation}
\end{lemma}

\begin{proof}
Let $\Omega_{t,r}$ and $\Rcal_\lambda$ be given by
\eqref{QQ-5'} and \eqref{QQ-5}, respectively. Multiplying
\eqref{eq:inhom-linear} by $\Rcal_\lambda$ and \eqref{QQ-4} by $\psi$
yields that for $(s,\rho)\in\Omega_{t,r}$,
\begin{equation}
\label{QQ-9}
\Rcal_\lambda G
=\p_s\left(
\Rcal_\lambda\p_s\psi-\psi\p_s\Rcal_\lambda
\right)
-\p_\rho\left(
\Rcal_\lambda\p_\rho\psi-\psi\p_\rho\Rcal_\lambda
\right).
\end{equation}
Analogously to the treatment of \eqref{QQ-3},
we then have
\begin{align}
&\int_1^t\int_{r-(t-s)}^{r+(t-s)}
\Rcal_\lambda(t,r;s,\rho)G(s,\rho)\dd\rho\dd s\notag\\
&=-\int_{r-t+1}^{r+t-1}
\left[
\Rcal_\lambda\p_s\psi-\psi\p_s\Rcal_\lambda
\right]_{s=1}\dd\rho\notag+\int_1^t
\left[
\Rcal_\lambda(\p_s\psi-\p_\rho\psi)
-\psi(\p_s-\p_\rho)\Rcal_\lambda
\right]_{\rho=r+(t-s)}\dd s\notag\\
&\quad+\int_1^t
\left[
\Rcal_\lambda(\p_s\psi+\p_\rho\psi)
-\psi(\p_s+\p_\rho)\Rcal_\lambda
\right]_{\rho=r-(t-s)}\dd s.
\label{QQ-10}
\end{align}
By the vanishing initial data  in \eqref{eq:inhom-linear}, one has
\begin{equation}
\label{QQ-11}
\int_{r-t+1}^{r+t-1}
\left[
\Rcal_\lambda\p_s\psi-\psi\p_s\Rcal_\lambda
\right]_{s=1}\dd\rho
=0.
\end{equation}
Moreover, it follows from \eqref{QQ-6} and \eqref{QQ-7} that
\begin{equation}
\label{QQ-12}
\begin{aligned}
&\int_1^t
\left[
\Rcal_\lambda(\p_s\psi-\p_\rho\psi)
-\psi(\p_s-\p_\rho)\Rcal_\lambda
\right]_{\rho=r+(t-s)}\dd s\\
&\quad +\int_1^t
\left[
\Rcal_\lambda(\p_s\psi+\p_\rho\psi)
-\psi(\p_s+\p_\rho)\Rcal_\lambda
\right]_{\rho=r-(t-s)}\dd s\\
&=2\psi(t,r).
\end{aligned}
\end{equation}
Combining \eqref{QQ-11}-\eqref{QQ-12} and \eqref{QQ-10} implies
\begin{equation}
\label{QQ-14}
2\psi(t,r)
=\int_1^t\int_{r-(t-s)}^{r+(t-s)}\Rcal_\lambda(t,r;s,\rho)G(s,\rho)\dd\rho\dd s,
\end{equation}
which proves
\eqref{eq:inhom-representation}.
\end{proof}
Collecting Lemmas~\ref{lem:homogeneous-representation} and
\ref{lem:inhomogeneous-representation} with the superposition principle,
we have the following representation formula.

\begin{proposition}[{\bf Representation formula}]
\label{prop:full-representation}
For $\psi_0$ and $\psi_1$ defined in \eqref{eq:psi-equation} and
$G\in C^1([1,\infty)\times\R)$, the solution of
\begin{equation}
\label{eq:full-linear}
\left\{
\begin{aligned}
&\p_t^2\psi-\p_r^2\psi-\f{\lambda(\lambda+1)}{t^2}\psi=G,
&&t\geq1,\quad r\in\R,\\
&(\psi,\p_t\psi)(1,r)=\ve(\psi_0,\psi_1)(r),
\end{aligned}
\right.
\end{equation}
is given by
\begin{equation}
\label{eq:full-representation}
\begin{aligned}
\psi(t,r)
={}&\f{\ve}{2}\psi_0(r+t-1)
+\f{\ve}{2}\psi_0(r-t+1)
+\ve c_\lambda\int_{r-t+1}^{r+t-1}
K_0(t,r-\rho)\psi_0(\rho)\dd\rho\\
&+\ve c_\lambda\int_{r-t+1}^{r+t-1}
K_1(t,r-\rho)\psi_1(\rho)\dd\rho
+c_\lambda\int_1^t\int_{r-(t-s)}^{r+(t-s)}
\Ecal_\lambda(t,s,r-\rho)G(s,\rho)\dd\rho\dd s.
\end{aligned}
\end{equation}
\end{proposition}
\begin{remark}
It is pointed out that Proposition~\ref{prop:full-representation} is closely related to
\cite[Theorem~1]{P2}. Indeed, setting $\tau=t-1$, $b=s-1$, $x=r$,
$y=\rho$, $\mu_1=0$ and $\nu^2=-\lambda(\lambda+1)$ in
\cite[Theorem~1]{P2} with $\delta=(2\lambda+1)^2$ formally yields
the kernel $\Ecal_\lambda$ and the representation formula
\eqref{eq:full-representation}.
However, \cite[Theorem~1]{P2} is established under the assumption of
$\nu^2\geq0$, whereas in the present problem
$\nu^2=-\lambda(\lambda+1)<0$. Therefore, it is necessary to derive
the corresponding representation formula for \eqref{eq:full-linear} explicitly.
\end{remark}

\section{Weighted  estimates of solutions for the homogeneous problem}
\label{sec:linear-estimates}
In this section, we establish some weighted spacetime pointwise estimates for the
solution $\psi_{\rm lin}$ of homogeneous  problem \eqref{eq:hom-linear}. Writing
$w_{\rm lin}(t,r)=\f{\psi_{\rm lin}(t,r)}{r}$ and
\[
\Acal(t,r)=1+t+r,\quad \Bcal(t,r)=1+t-r,
\]
where 
\begin{equation*}
\begin{aligned}
\psi_{\rm lin}(t,r)=&\f{\ve}{2}\psi_0(r+t-1)
+\f{\ve}{2}\psi_0(r-t+1)
+\ve c_\lambda\int_{r-t+1}^{r+t-1}
K_0(t,r-\rho)\psi_0(\rho)\dd\rho\\
&+\ve c_\lambda\int_{r-t+1}^{r+t-1}
K_1(t,r-\rho)\psi_1(\rho)\dd\rho
\end{aligned}
\end{equation*} 
which comes from the expression \eqref{eq:full-representation}.
We next prove that for $1<\kappa<2-\lambda$,
\begin{equation}\label{YHC-05}
|w_{\rm lin}(t,r)|
+\f{1}{1+r}|\p_r(rw_{\rm lin})(t,r)|
\lesssim
\ve\Acal(t,r)^{-1}\Bcal(t,r)^{1-\kappa},
\quad t\geq1,\quad 0\leq r\leq t.
\end{equation}

To this end, we first give two technical lemmas concerning  the
Gauss hypergeometric function and the kernels appeared in
\eqref{eq:hom-representation}.

\begin{lemma}
\label{lem:Phi}
Let $0<\lambda<1/2$ and
\begin{equation}\label{QQ-19'}
\mathcal F_\lambda(Q)
=
Q^\lambda
{}_2F_1\left(
-\lambda,-\lambda;1;1-Q^{-1}
\right),
\quad Q\geq1.
\end{equation}
Then, for $j=0,1,2$,
\begin{equation}\label{QQ-19}
\big|
\mathcal F_\lambda^{(j)}(Q)
\big|
\leq
C_\lambda Q^{\lambda-j},
\quad Q\geq1.
\end{equation}
\end{lemma}

\begin{proof}
We consider separately the regions $1\leq Q\leq2$ and $Q\geq2$.

For $1\leq Q\leq2$, one has
\[
0\leq1-Q^{-1}\leq\f12.
\]
Since the Gauss hypergeometric function ${}_2F_1(-\lambda,-\lambda;1;z)$ is analytic for $|z|<1$ (see \cite[Section~15.2]{OLBC}),
$\mathcal F_\lambda$ is smooth on $[1,2]$. Hence, for $j=0,1,2$,
\begin{equation}\label{QQ-15}
\big|
\mathcal F_\lambda^{(j)}(Q)
\big|
\leq
C_\lambda Q^{\lambda-j},
\quad 1\leq Q\leq2.
\end{equation}

We next assume $Q\geq2$ and set $y=Q^{-1}$, so that
$0<y\leq1/2$. It follows from
\cite[Section ~15.8]{OLBC} that
\[
\begin{aligned}
{}_2F_1(-\lambda,-\lambda;1;1-y)
={}&
C_{0}(\lambda)
{}_2F_1(-\lambda,-\lambda;-2\lambda;y)+
C_{1}(\lambda)
y^{1+2\lambda}
{}_2F_1(1+\lambda,1+\lambda;2+2\lambda;y),
\end{aligned}
\]
where $C_0(\lambda)$ and $C_1(\lambda)$ are some constants depending
on $\lambda$.
Due to $0<\lambda<1/2$, all the hypergeometric functions
above are well defined.

Set
\[
A_\lambda(y)={}_2F_1(-\lambda,-\lambda;-2\lambda;y),
\qquad
B_\lambda(y)={}_2F_1(1+\lambda,1+\lambda;2+2\lambda;y).
\]
Both functions are analytic for $|y|<1$. Therefore,
\begin{equation}\label{QQ-16}
\mathcal F_\lambda(Q)
=
C_{0}(\lambda)Q^\lambda A_\lambda(Q^{-1})
+
C_{1}(\lambda)Q^{-1-\lambda}B_\lambda(Q^{-1})=C_{0}(\lambda)\mathcal F_{\lambda,0}(Q)+C_{1}(\lambda)\mathcal F_{\lambda,1}(Q),
\end{equation}
where
\[
\sup_{0\leq y\leq1/2}
\big(
|A_\lambda^{(m)}(y)|
+
|B_\lambda^{(m)}(y)|
\big)
\leq C_\lambda,
\quad m=0,1,2.
\]
A direct differentiation gives
\[
\mathcal F_{\lambda,0}'(Q)
=\lambda Q^{\lambda-1}A_\lambda(Q^{-1})
-Q^{\lambda-2}A_\lambda'(Q^{-1}),
\]
and
\[
\begin{aligned}
\mathcal F_{\lambda,0}''(Q)
={}&\lambda(\lambda-1)Q^{\lambda-2}A_\lambda(Q^{-1})
+2(1-\lambda)Q^{\lambda-3}A_\lambda'(Q^{-1})+
Q^{\lambda-4}A_\lambda''(Q^{-1}).
\end{aligned}
\]
Then
\begin{equation}\label{QQ-17}
\big|
\mathcal F_{\lambda,0}^{(j)}(Q)
\big|
\leq
C_\lambda Q^{\lambda-j},
\quad j=0,1,2.
\end{equation}
Similarly, for $Q\geq2$,
\begin{equation}\label{QQ-18}
\big|
\mathcal F_{\lambda,1}^{(j)}(Q)
\big|
\leq
C_\lambda Q^{\lambda-j},
\quad j=0,1,2.
\end{equation}
Hence, \eqref{QQ-19} follows  from \eqref{QQ-15} and \eqref{QQ-17}-\eqref{QQ-18} directly.
\end{proof}

\begin{lemma}
\label{lem:kernel-pointwise}
Let
\begin{equation}\label{QQ-21'}
D(t,s,y)=(t+s)^2-y^2,
\quad
Q(t,s,y)=\f{D(t,s,y)}{4ts}.
\end{equation}
For $1\leq s\leq t$ and $|y|\leq t-s$, one has
\begin{equation}\label{QQ-20}
|\Ecal_\lambda(t,s,y)|
\lesssim Q(t,s,y)^\lambda,
\quad
|\p_y\Ecal_\lambda(t,s,y)|
\lesssim
Q(t,s,y)^\lambda\f{|y|}{D(t,s,y)}.
\end{equation}
Moreover, for $K_0$ and $K_1$ defined in \eqref{eq:K01},
$|y|\leq t-1$ and $j=0,1$,
\begin{equation}\label{QQ-21}
|K_j(t,y)|
\lesssim Q(t,1,y)^\lambda,
\quad
|\p_yK_j(t,y)|
\lesssim
Q(t,1,y)^\lambda\f{|y|}{D(t,1,y)}.
\end{equation}
\end{lemma}

\begin{proof}
Note that  for $1\leq s\leq t$ and $|y|\leq t-s$,
\[
D(t,s,y)\geq(t+s)^2-(t-s)^2=4ts.
\]
Then
\[
Q(t,s,y)=\f{D(t,s,y)}{4ts}\geq1.
\]
By the definition of
$\mathcal F_\lambda$ in \eqref{QQ-19'}, one arrives at
\begin{equation}\label{QQ-22}
\Ecal_\lambda(t,s,y)
=
4^\lambda\mathcal F_\lambda(Q(t,s,y)).
\end{equation}
Thus,  \eqref{QQ-19} with $j=0$ yields that
\begin{equation}\label{QQ-23'}
|\Ecal_\lambda(t,s,y)|
\lesssim
Q(t,s,y)^\lambda.
\end{equation}
Differentiating $Q$ with respect to $y$, we obtain
\begin{equation}\label{QQ-23}
\p_yQ(t,s,y)
=
-\f{y}{2ts}
=
-\f{2y}{D(t,s,y)}Q(t,s,y).
\end{equation}
It follows from \eqref{QQ-19}, \eqref{QQ-22} and \eqref{QQ-23} that
\[
\begin{aligned}
|\p_y\Ecal_\lambda(t,s,y)|
&=
4^\lambda
\left|
\mathcal F_\lambda'(Q(t,s,y))\p_yQ(t,s,y)
\right|
\lesssim
Q(t,s,y)^\lambda\f{|y|}{D(t,s,y)}.
\end{aligned}
\]
This, together with \eqref{QQ-23'}, yields \eqref{QQ-20}.

The estimates for $K_1$ in \eqref{QQ-21} follow immediately from
\eqref{QQ-20} by setting $s=1$. We now investigate $K_0$.
Differentiating $Q=D/(4ts)$ with respect to $s$ gives
\begin{equation}\label{QQ-24}
\p_sQ(t,s,y)
=Q(t,s,y)
\left(
\f{2(t+s)}{D(t,s,y)}-\f1s
\right).
\end{equation}
For $s=1$ and $|y|\leq t-1$,
\[
D(t,1,y)
\geq
(t+1)^2-(t-1)^2
=4t.
\]
Therefore,
\[
\left|
\f{\p_sQ(t,1,y)}{Q(t,1,y)}
\right|
\leq
\f{2(t+1)}{D(t,1,y)}+1
\leq2,
\]
and then
\begin{equation}\label{QQ-25}
|\p_sQ(t,1,y)|
\lesssim
Q(t,1,y).
\end{equation}
By \eqref{QQ-22} and \eqref{eq:K01}, we have
\begin{equation}\label{QQ-26}
K_0(t,y)
=-4^\lambda
\mathcal F_\lambda'(Q(t,1,y))\p_sQ(t,1,y).
\end{equation}
Combining \eqref{QQ-19} and \eqref{QQ-25}-\eqref{QQ-26} implies
\begin{equation}\label{QQ-27'}
|K_0(t,y)|
\lesssim
Q(t,1,y)^{\lambda-1}|\p_sQ(t,1,y)|
\lesssim
Q(t,1,y)^\lambda.
\end{equation}
We next turn to estimate $\p_yK_0$. From \eqref{QQ-23}, it holds that
\begin{equation}\label{QQ-27}
\p_{sy}^2Q(t,s,y)
=\f{y}{2ts^2},
\quad
|\p_{sy}^2Q(t,1,y)|
=2Q(t,1,y)\f{|y|}{D(t,1,y)}.
\end{equation}
Differentiating \eqref{QQ-26} with respect to $y$, one can arrive at
\begin{equation}\label{QQ-28}
\begin{aligned}
\p_yK_0(t,y)
=-4^\lambda\bigl[
&\mathcal F_\lambda''(Q(t,1,y))
\p_sQ(t,1,y)\p_yQ(t,1,y)+
\mathcal F_\lambda'(Q(t,1,y))
\p_{sy}^2Q(t,1,y)
\bigr].
\end{aligned}
\end{equation}
Thus, using \eqref{QQ-19}, \eqref{QQ-23}, \eqref{QQ-25} and
\eqref{QQ-27} yields
\[
\begin{aligned}
|\p_yK_0(t,y)|
&\lesssim
Q(t,1,y)^{\lambda-2}
|\p_sQ(t,1,y)||\p_yQ(t,1,y)|
+Q(t,1,y)^{\lambda-1}|\p_{sy}^2Q(t,1,y)|\\
&\lesssim
Q(t,1,y)^\lambda
\f{|y|}{D(t,1,y)}.
\end{aligned}
\]
This, together with \eqref{QQ-27'}, gives \eqref{QQ-21}.
\end{proof}

We introduce the following weighted function space, which will be used
to establish the required estimates for both the solutions of homogeneous and inhomogeneous problems.
Motivated by \cite{Rei1}, let $X_\kappa$ denote the space
of even continuous function $w=w(t,r)$ satisfying
$\partial_r(rw)\in C([1,\infty)\times\mathbb R)$
and
$\supp w\subset
\{(t,r)\in[1,\infty)\times\R:|r|\leq t\},$
equipped with the norm
\begin{equation}\label{eq:Xkappa}
\begin{aligned}
\|w\|_{X_\kappa}
=\sup_{\substack{t\geq1\\0\leq r\leq t}}
\big\{
&\Acal(t,r)\Bcal(t,r)^{\kappa-1}|w(t,r)|+
\f{\Acal(t,r)\Bcal(t,r)^{\kappa-1}}{1+r}
\left|\p_r(rw)(t,r)\right|
\big\}.
\end{aligned}
\end{equation}

Based on Lemmas~\ref{lem:Phi} and \ref{lem:kernel-pointwise}, we
prove the following estimate.

\begin{proposition}
\label{prop:linear}
Let $1<\kappa<2-\lambda$. Then, for $t\geq1$ and $0\leq r\leq t$,
\begin{equation}
\label{eq:linear-weighted}
|w_{\rm lin}(t,r)|
+\f{1}{1+r}
\left|
\p_r(rw_{\rm lin})(t,r)
\right|
\lesssim
\ve\Acal(t,r)^{-1}\Bcal(t,r)^{1-\kappa}.
\end{equation}
In particular,
\begin{equation}\label{QQ-29}
\|w_{\rm lin}\|_{X_\kappa}\lesssim\ve.
\end{equation}
\end{proposition}
\begin{proof}
It follows from \eqref{eq:hom-representation} that
\begin{equation}\label{QQ-29'}
\begin{aligned}
\psi_{\rm lin}(t,r)
={}&
\f{\ve}{2}\psi_0(r+t-1)
+\f{\ve}{2}\psi_0(r-t+1)+
\ve c_\lambda
\sum_{j=0}^1
\int_{r-t+1}^{r+t-1}
K_j(t,r-\rho)\psi_j(\rho)\dd\rho,
\end{aligned}
\end{equation}
where $\supp\psi_j\subset[-1,1]$, $\psi_j$ is odd, and
$K_j(t,\cdot)$ is even for $j=0,1$.

It is asserted that
\begin{equation}\label{ass1}
|\psi_{\rm lin}(t,r)|
\lesssim
\ve r\Acal(t,r)^{-1}\Bcal(t,r)^{1-\kappa}
\end{equation}
and
\begin{equation}\label{ass2}
|\p_r\psi_{\rm lin}(t,r)|
\lesssim
\ve(1+r)\Acal(t,r)^{-1}\Bcal(t,r)^{1-\kappa},
\end{equation}
where $1<\kappa<2-\lambda$,  $t\geq1$ and $0\leq r\leq t$.

To prove \eqref{ass1} and \eqref{ass2}, we consider separately the
cases of $\Bcal(t,r)\geq3$ and $1\leq\Bcal(t,r)<3$.
\vskip 0.2 true cm

{\bf Case 1.  $\Bcal(t,r)\geq3$ }

\vskip 0.2 true cm

Due to $t-r\geq2$, one has $r-t+1\leq-1$ and $r+t-1\geq1$.
Then by the support condition of the initial data, $\psi_{\rm lin}(t,r)$ can be written as
\[
\psi_{\rm lin}(t,r)
=\ve c_\lambda
\sum_{j=0}^1
\int_{-1}^1
K_j(t,r-\rho)\psi_j(\rho)\dd\rho.
\]
Then using the oddness of $\psi_j$ and the evenness of $K_j$, one obtains
\begin{equation}\label{QQ-30}
\begin{aligned}
\psi_{\rm lin}(t,r)
&=\ve c_\lambda
\sum_{j=0}^1
\int_0^1
\left[
K_j(t,r-q)-K_j(t,r+q)
\right]\psi_j(q)\dd q\\
&=\ve c_\lambda
\sum_{j=0}^1
\int_0^1
\left[
K_j(t,|r-q|)-K_j(t,r+q)
\right]\psi_j(q)\dd q\\
&=-\ve c_\lambda
\sum_{j=0}^1
\int_0^1
\int_{|r-q|}^{r+q}
\p_yK_j(t,y)\dd y\,
\psi_j(q)\dd q.
\end{aligned}
\end{equation}
By $q\in[0,1]$, $0\leq r\leq t$ and $|r-q|\leq y\leq r+q$, we arrive at
\[
t+1-y
\geq
\Bcal(t,r)-q
\geq
\Bcal(t,r)-1
\geq
\f23\Bcal(t,r)
\]
and
\[
t+1+y
\geq
t+1
\geq
\f12(1+t+r)=\f12\Acal(t,r),
\]
which yield
\[
D(t,1,y)=(t+1-y)(t+1+y)
\gtrsim
\Acal(t,r)\Bcal(t,r).
\]
On the other hand,  $\Bcal(t,r)\geq3$ and $0\leq r\leq t$ imply that
\[
\begin{aligned}
D(t,1,y)=(t+1)^2-y^2
&\leq
(t+1)^2-(r-q)^2=
\bigl(\Bcal(t,r)+q\bigr)
\bigl(\Acal(t,r)-q\bigr)
\lesssim
\Acal(t,r)\Bcal(t,r)
\end{aligned}
\]
 and
\[
\Acal(t,r)=1+t+r\leq2t.
\]
Therefore,
\begin{equation}\label{QQ-31}
D(t,1,y)\sim\Acal(t,r)\Bcal(t,r),\quad
Q(t,1,y)=\f{D(t,1,y)}{4t}
\lesssim\f{\Acal(t,r)\Bcal(t,r)}{\Acal(t,r)}=\Bcal(t,r).
\end{equation}
Due to $\Bcal(t,r)\geq3$, one has
\[
0\leq y\leq r+q\leq r+1\leq t-1.
\]
Then it follows from \eqref{QQ-21} and \eqref{QQ-31} that for  $j=0,1$,
\begin{equation}\label{QQ-31''}
\begin{aligned}
|\p_yK_j(t,y)|
&\lesssim
Q(t,1,y)^\lambda
\f{y}{D(t,1,y)}\lesssim
\Acal(t,r)^{-1}\Bcal(t,r)^{\lambda-1}y.
\end{aligned}
\end{equation}
Substituting \eqref{QQ-31''} into \eqref{QQ-30} yields
\begin{equation}\label{QQ-32}
\begin{aligned}
|\psi_{\rm lin}(t,r)|
&\lesssim
\ve\Acal(t,r)^{-1}\Bcal(t,r)^{\lambda-1}
\sum_{j=0}^1
\int_0^1
\int_{|r-q|}^{r+q}
y\dd y\,|\psi_j(q)|\dd q\\
&=2\ve r\Acal(t,r)^{-1}\Bcal(t,r)^{\lambda-1}
\sum_{j=0}^1
\int_0^1q|\psi_j(q)|\dd q\\
&\lesssim
\ve r\Acal(t,r)^{-1}\Bcal(t,r)^{\lambda-1}.
\end{aligned}
\end{equation}
On the other hand, differentiating the first equality in \eqref{QQ-30} with respect
to $r$ and using \eqref{QQ-31''}, one can obtain
\[
\begin{aligned}
|\p_r\psi_{\rm lin}(t,r)|
&\lesssim
\ve\sum_{j=0}^1\int_0^1
\Bigl(
|\p_yK_j(t,r-q)|
+|\p_yK_j(t,r+q)|
\Bigr)
|\psi_j(q)|\dd q\\
&\lesssim
\ve\Acal(t,r)^{-1}\Bcal(t,r)^{\lambda-1}
\sum_{j=0}^1
\int_0^1
\bigl(|r-q|+r+q\bigr)|\psi_j(q)|\dd q\\
&\lesssim
\ve\Acal(t,r)^{-1}\Bcal(t,r)^{\lambda-1}
\sum_{j=0}^1
\int_0^1 2
\bigl(r+1\bigr)|\psi_j(q)|\dd q\\
&\lesssim
\ve(1+r)
\Acal(t,r)^{-1}\Bcal(t,r)^{\lambda-1}.
\end{aligned}
\]
For $\kappa<2-\lambda$ and $\Bcal(t,r)\geq3$, we then have
\begin{equation}\label{QQ-32'}
\begin{aligned}
|\psi_{\rm lin}(t,r)|
\lesssim
\ve r\Acal(t,r)^{-1}\Bcal(t,r)^{1-\kappa},
\quad
|\p_r\psi_{\rm lin}(t,r)|
\lesssim
\ve(1+r)\Acal(t,r)^{-1}\Bcal(t,r)^{1-\kappa}.
\end{aligned}
\end{equation}
Therefore, \eqref{ass1} and \eqref{ass2} hold for
$\Bcal(t,r)\geq3$.

\vskip 0.2 true cm

{\bf Case 2. $1\leq\Bcal(t,r)<3$.}

\vskip 0.2 true cm

Differentiating \eqref{QQ-29'} with respect to $r$ yields
\begin{equation}\label{YHC-07}
\begin{aligned}
&\p_r\psi_{\rm lin}(t,r)
=\f{\ve}{2}\psi_0'(r+t-1)
+\f{\ve}{2}\psi_0'(r-t+1)\\
&+\ve c_\lambda
\sum_{j=0}^1
\big[K_j(t,1-t)\psi_j(r+t-1)
-K_j(t,t-1)\psi_j(r-t+1)
+\int_{r-t+1}^{r+t-1}
\p_yK_j(t,r-\rho)\psi_j(\rho)\dd\rho
\big].
\end{aligned}
\end{equation}
Due to $0\leq r\leq t$ and $1\leq\Bcal(t,r)<3$, we have
\[
t\leq\Acal(t,r)=1+t+r\leq3t
\]
and
\[
1+r
\leq
\Acal(t,r)
=
\Bcal(t,r)+2r
<
3+2r
\leq
3(1+r).
\]
Thus,
\begin{equation}\label{QQ-34}
\Acal(t,r)\sim t\sim1+r.
\end{equation}

We first treat the terms involving $K_j(t,1-t)$ and
$K_j(t,t-1)$ in \eqref{YHC-07}. It follows from Lemma \ref{lem:kernel-pointwise} and the evenness of $K_j$ that for $j=0,1$,
\[
|K_j(t,1-t)|
+|K_j(t,t-1)|\lesssim Q(t,1,|t-1|)^{\lambda}
=1.
\]
As for the  term $\p_yK_j(t,r-\rho)$, by \eqref{QQ-29'}, Lemma \ref{lem:kernel-pointwise} and the support condition on the initial
data, one has
\[
\begin{aligned}
D(t,1,r-\rho)
&=(t+1)^2-(r-\rho)^2=
\bigl(\Bcal(t,r)+\rho\bigr)
\bigl(\Acal(t,r)-\rho\bigr)
\end{aligned}
\]
and
\[
\rho\in[-1,1]\cap[r-t+1,r+t-1].
\]
Due to
$
\rho\geq r-t+1=2-\Bcal(t,r)
$
and $|\rho|\leq1$, it follows from $1\leq\Bcal(t,r)<3$ and $\Acal(t,r)\geq2$ that
\[
2\leq\Bcal(t,r)+\rho<4, \quad \f12\Acal(t,r)
\leq
\Acal(t,r)-\rho
\leq
\f32\Acal(t,r).
\]
Then
\begin{equation}\label{QQ-33'}
D(t,1,r-\rho)\sim\Acal(t,r).
\end{equation}
Combining \eqref{QQ-33'} with \eqref{QQ-34} derives
\begin{equation}\label{QQ-33}
Q(t,1,r-\rho)
=\f{D(t,1,r-\rho)}{4t}
\sim1.
\end{equation}
Moreover, by
$\rho\in[r-t+1,r+t-1]$, one has
\[
|r-\rho|\leq t-1.
\]
Thus, \eqref{QQ-21}, \eqref{QQ-33'} and \eqref{QQ-33} imply that
for $j=0,1$,
\[
\begin{aligned}
|\p_yK_j(t,r-\rho)|
&\lesssim
Q(t,1,r-\rho)^\lambda
\f{|r-\rho|}{D(t,1,r-\rho)}\lesssim
\f{r+|\rho|}{\Acal(t,r)}
\lesssim
\f{1+r}{\Acal(t,r)}\Bcal(t,r)^{1-\kappa}.
\end{aligned}
\]
Using the above estimates in the expression of
$\p_r\psi_{\rm lin}(t,r)$, we obtain
\[
\begin{aligned}
|\p_r\psi_{\rm lin}(t,r)|
\lesssim&
\ve\bigl(
|\psi_0'(r+t-1)|
+|\psi_0'(r-t+1)|
\bigr)+
\ve\sum_{j=0}^1
\bigl(
|\psi_j(r+t-1)|
+|\psi_j(r-t+1)|
\bigr)\\
&+\ve\f{1+r}{\Acal(t,r)}
\sum_{j=0}^1
\int_{[r-t+1,r+t-1]\cap[-1,1]}
|\psi_j(\rho)|\dd\rho.
\end{aligned}
\]
Since $\psi_0$, $\psi_1$, and $\psi_0'$ are bounded, and
$\Acal(t,r)\sim1+r$ by \eqref{QQ-34}, it follows that
\begin{equation}\label{QQ-35'}
|\p_r\psi_{\rm lin}(t,r)|
\lesssim
\ve(1+r)\Acal(t,r)^{-1}
\Bcal(t,r)^{1-\kappa}.
\end{equation}

We next estimate $\psi_{\rm lin}(t,r)$. If
$r\geq1$, one has that for $1\leq\Bcal(t,r)<3$,
\[
\Acal(t,r)=\Bcal(t,r)+2r\sim r.
\]
Using \eqref{QQ-29'} and \eqref{QQ-33}, we arrive at
\[
|\psi_{\rm lin}(t,r)|
\lesssim
\ve
\lesssim
\ve\f{r}{\Acal(t,r)}\Bcal(t,r)^{1-\kappa}.
\]
If $0\leq r\leq1$, then $\Bcal(t,r)<3$ gives $t\leq3$.
Since $\psi_{\rm lin}(t,\cdot)$ is odd, $\psi_{\rm lin}(t,0)=0$ holds.
Hence,
\begin{equation}\label{QQ-35}
\begin{aligned}
|\psi_{\rm lin}(t,r)|
&=
\left|
\psi_{\rm lin}(t,r)-\psi_{\rm lin}(t,0)
\right|\leq
\int_0^r
|\p_q\psi_{\rm lin}(t,q)|\dd q.
\end{aligned}
\end{equation}
For $0\leq q\leq r$,  \eqref{QQ-32'} and \eqref{QQ-35'} give
\[
|\p_q\psi_{\rm lin}(t,q)|
\lesssim
\ve.
\]
This, together with \eqref{QQ-35}, implies
\[
|\psi_{\rm lin}(t,r)|
\lesssim
\ve r.
\]
For $t\leq3$, $0\leq r\leq1$, and
$1\leq\Bcal(t,r)<3$, one can obtain
\[
\Acal(t,r)\sim1,
\quad
\Bcal(t,r)^{1-\kappa}\sim1.
\]
Therefore,
\[
|\psi_{\rm lin}(t,r)|
\lesssim
\ve r\Acal(t,r)^{-1}\Bcal(t,r)^{1-\kappa}.
\]
Collecting the estimates in Cases 1--2, \eqref{ass1} and
\eqref{ass2} hold for $t\geq1$ and $0\leq r\leq t$.

For $r>0$,
using \eqref{ass1},
\eqref{ass2} and
\[
w_{\rm lin}(t,r)=\f{\psi_{\rm lin}(t,r)}{r},
\quad
\p_r(rw_{\rm lin})(t,r)=\p_r\psi_{\rm lin}(t,r),
\]
we get
\[
\begin{aligned}
|w_{\rm lin}(t,r)|
+
\f{1}{1+r}
\left|
\p_r(rw_{\rm lin})(t,r)
\right|
&=
\f{|\psi_{\rm lin}(t,r)|}{r}
+
\f{|\p_r\psi_{\rm lin}(t,r)|}{1+r}\lesssim
\ve\Acal(t,r)^{-1}\Bcal(t,r)^{1-\kappa}.
\end{aligned}
\]
For $r=0$, the oddness of $\psi_{\rm lin}(t,\cdot)$ gives
\[
w_{\rm lin}(t,0)=
\lim_{r\to0}
\f{\psi_{\rm lin}(t,r)-\psi_{\rm lin}(t,0)}{r}=\p_r\psi_{\rm lin}(t,0),
\quad
\p_r(rw_{\rm lin})(t,0)=\p_r\psi_{\rm lin}(t,0).
\]
It follows from \eqref{ass2} with $r=0$ that
\[
\begin{aligned}
|w_{\rm lin}(t,0)|
+
\left|
\p_r(rw_{\rm lin})(t,0)
\right|
&=
2|\p_r\psi_{\rm lin}(t,0)|\lesssim
\ve\Acal(t,0)^{-1}\Bcal(t,0)^{1-\kappa}.
\end{aligned}
\]
Therefore, \eqref{eq:linear-weighted} follows. By the definition of
$X_\kappa$ \eqref{eq:Xkappa}, we further obtain
\[
\|w_{\rm lin}\|_{X_\kappa}\lesssim\ve,
\]
which proves \eqref{QQ-29} and completes the proof of
Proposition~\ref{prop:linear}.
\end{proof}

\section{Weighted estimates of solutions for the inhomogeneous problem}
\label{sec:nonlinear-Riemann-operator}

In this section, we establish some weighted pointwise estimates for the
solutions of the inhomogeneous problem
associated with the nonlinear source $G_w$
introduced in Subsection~\ref{liou}. With
$\lambda=(\mu-2)/2$ and $\beta=\mu(p-1)/2$, we choose $\kappa$ such that
\begin{equation}
\label{eq:kappa-range}
\f{2-\beta}{p-1}
<\kappa
<\min\{p+\beta-1,\,2-\lambda\}.
\end{equation}
It is worth noting that for
$\f{14}{5}\leq\mu<3$ and
$\f{5}{3}<p\leq1+\f{2}{\mu}$, the choice of $\kappa$ in
\eqref{eq:kappa-range} can be realized. Indeed,
\[
2-\lambda-\f{2-\beta}{p-1}
=\f{3p-5}{p-1}>0,
\quad
p+\beta-1-\f{2-\beta}{p-1}>0.
\]
Moreover, $\f{2-\beta}{p-1}>1$, and hence $\kappa>1$ can be chosen.

Throughout this section, let $w\in X_\kappa$, where $X_\kappa$ is
the function space introduced in Section~\ref{sec:linear-estimates}
with norm \eqref{eq:Xkappa}. In view of \eqref{eq:inhom-representation}, for $r>0$ we define
\begin{equation}
\label{qq:method-nonlinear}
r(\Tcal w)(t,r)
=c_\lambda
\int_1^t
\int_{r-(t-s)}^{r+(t-s)}
\Ecal_\lambda(t,s,r-\rho)
G_w(s,\rho)\dd\rho\dd s.
\end{equation}
To facilitate the derivation of the weighted estimates, we rewrite
\eqref{qq:method-nonlinear} in the following form.

\begin{lemma}
\label{lem:nonlinear-representation}
For $t\geq1$ and $0<r\leq t$, it holds that
\begin{equation}
\label{qq:40}
\begin{aligned}
c_\lambda^{-1}r(\Tcal w)(t,r)
={}&
\int_1^t
\int_{|t-s-r|}^{t-s+r}
\Ecal_\lambda(t,s,r-q)
G_w(s,q)\dd q\dd s\\
&+\int_1^{t-r}
\int_0^{t-s-r}
\left[
\Ecal_\lambda(t,s,r-q)
-\Ecal_\lambda(t,s,r+q)
\right]
G_w(s,q)\dd q\dd s.
\end{aligned}
\end{equation}
Moreover,
\begin{equation}
\label{qq:42}
\begin{aligned}
c_\lambda^{-1}\p_r(r\Tcal w)(t,r)
={}&
4^\lambda
\int_1^t
\left[
G_w(s,t-s+r)
+G_w(s,t-s-r)
\right]\dd s\\
&+\int_1^t
\int_{|t-s-r|}^{t-s+r}
\p_y\Ecal_\lambda(t,s,r-q)
G_w(s,q)\dd q\dd s\\
&+\int_1^{t-r}
\int_0^{t-s-r}
\left[
\p_y\Ecal_\lambda(t,s,r-q)
-\p_y\Ecal_\lambda(t,s,r+q)
\right]
G_w(s,q)\dd q\dd s.
\end{aligned}
\end{equation}
Here the last terms in
\eqref{qq:40} and \eqref{qq:42} vanish when $t-r\leq1$.
\end{lemma}

\begin{proof}
Note that when $t-s\leq r$,  $r-(t-s)=|t-s-r|$ holds. Then
\begin{equation}\label{QQ-37}
\int_{r-(t-s)}^{r+(t-s)}
\Ecal_\lambda(t,s,r-\rho)G_w(s,\rho)\dd\rho
=\int_{|t-s-r|}^{t-s+r}
\Ecal_\lambda(t,s,r-\rho)G_w(s,\rho)\dd\rho.
\end{equation}
When $t-s>r$,  by the change of variable $\rho=-q$ and the oddness of $G_w(s,\cdot)$, one can obtain
\begin{equation}\label{QQ-38}
\begin{aligned}
&\int_{r-(t-s)}^{r+(t-s)}
\Ecal_\lambda(t,s,r-\rho)G_w(s,\rho)\dd\rho\\
&=\int_{r-(t-s)}^0
\Ecal_\lambda(t,s,r-\rho)G_w(s,\rho)\dd\rho
+\int_0^{t-s+r}
\Ecal_\lambda(t,s,r-\rho)G_w(s,\rho)\dd\rho\\
&=-\int_0^{t-s-r}
\Ecal_\lambda(t,s,r+q)G_w(s,q)\dd q
+\int_0^{t-s+r}
\Ecal_\lambda(t,s,r-q)G_w(s,q)\dd q\\
&=\int_{t-s-r}^{t-s+r}
\Ecal_\lambda(t,s,r-q)G_w(s,q)\dd q+
\int_0^{t-s-r}
\left[
\Ecal_\lambda(t,s,r-q)
-\Ecal_\lambda(t,s,r+q)
\right]
G_w(s,q)\dd q.
\end{aligned}
\end{equation}
Combining the two cases and integrating with respect to $s$ over
$[1,t]$ yield \eqref{qq:40}.

We next turn to deal with $\p_r(r\Tcal w)(t,r)$. It follows from Leibniz's rule that
\[
\begin{aligned}
&\p_r
\int_{r-(t-s)}^{r+(t-s)}
\Ecal_\lambda(t,s,r-\rho)G_w(s,\rho)\dd\rho\\
&=\Ecal_\lambda(t,s,-(t-s))G_w(s,r+t-s)
-\Ecal_\lambda(t,s,t-s)G_w(s,r-t+s)\\
&\quad+
\int_{r-(t-s)}^{r+(t-s)}
\p_y\Ecal_\lambda(t,s,r-\rho)G_w(s,\rho)\dd\rho\\
&=4^\lambda
\left[
G_w(s,r+t-s)-G_w(s,r-t+s)
\right]+
\int_{r-(t-s)}^{r+(t-s)}
\p_y\Ecal_\lambda(t,s,r-\rho)G_w(s,\rho)\dd\rho\\
&=4^\lambda
\left[
G_w(s,t-s+r)+G_w(s,t-s-r)
\right]+
\int_{r-(t-s)}^{r+(t-s)}
\p_y\Ecal_\lambda(t,s,r-\rho)G_w(s,\rho)\dd\rho,
\end{aligned}
\]
where $\Ecal_\lambda(t,s,\pm(t-s))=4^\lambda$ and the oddness of
$G_w(s,\cdot)$ have been used.
Replacing $\Ecal_\lambda$ by $\p_y\Ecal_\lambda$ in
\eqref{QQ-37}--\eqref{QQ-38}, for $t-s\leq r$, one has
\[
\int_{r-(t-s)}^{r+(t-s)}
\p_y\Ecal_\lambda(t,s,r-\rho)G_w(s,\rho)\dd\rho
=\int_{|t-s-r|}^{t-s+r}
\p_y\Ecal_\lambda(t,s,r-\rho)G_w(s,\rho)\dd\rho,
\]
while for $t-s>r$, we can arrive at
\[
\begin{aligned}
&\int_{r-(t-s)}^{r+(t-s)}
\p_y\Ecal_\lambda(t,s,r-\rho)G_w(s,\rho)\dd\rho\\
&=
-\int_0^{t-s-r}
\p_y\Ecal_\lambda(t,s,r+q)G_w(s,q)\dd q
+
\int_0^{t-s+r}
\p_y\Ecal_\lambda(t,s,r-q)G_w(s,q)\dd q\\
&=
\int_{t-s-r}^{t-s+r}
\p_y\Ecal_\lambda(t,s,r-q)G_w(s,q)\dd q
+
\int_0^{t-s-r}
\left[
\p_y\Ecal_\lambda(t,s,r-q)
-
\p_y\Ecal_\lambda(t,s,r+q)
\right]
G_w(s,q)\dd q.
\end{aligned}
\]
Therefore, integration with respect to $s$ over $[1,t]$ gives
\eqref{qq:42}.
\end{proof}

Next, we establish the estimates for $\Ecal_\lambda$ and
$\p_y\Ecal_\lambda$ appearing in \eqref{qq:40} and \eqref{qq:42}.
For this purpose, set
\begin{equation}
\label{qq:44}
A_0=t+r,\quad B_0=t-r,\quad
\xi=s+q,\quad \eta=q-s,
\end{equation}
where $0\leq q\leq s$.

\begin{lemma}
\label{lem:characteristic-kernel}
If $|t-s-r|\leq q\leq t-s+r$, then
\begin{equation}
\label{qq:45}
|\Ecal_\lambda(t,s,r-q)|
\lesssim
\xi^\lambda(\xi-\eta)^{-\lambda},
\quad
|\p_y\Ecal_\lambda(t,s,r-q)|
\lesssim
\xi^{\lambda-1}(\xi-\eta)^{-\lambda}.
\end{equation}
If $1\leq s\leq t-r$ and $0\leq q\leq\min\{s,t-s-r\}$, then
\begin{equation}
\label{qq:46}
\begin{aligned}
&
\left|
\Ecal_\lambda(t,s,r-q)
-
\Ecal_\lambda(t,s,r+q)
\right|
\lesssim
\f{rqB_0^\lambda}{A_0(B_0-\eta)}
(\xi-\eta)^{-\lambda},\\
&
|\p_y\Ecal_\lambda(t,s,r-q)|
+
|\p_y\Ecal_\lambda(t,s,r+q)|
\lesssim
\f{(r+q)B_0^\lambda}{A_0(B_0-\eta)}
(\xi-\eta)^{-\lambda}.
\end{aligned}
\end{equation}
\end{lemma}

\begin{proof}
For $|t-s-r|\leq q\leq t-s+r$, it follows from \eqref{qq:44} and
\eqref{QQ-21'} that
\[
B_0\leq\xi\leq A_0,
\quad
D(t,s,r-q)=(B_0+\xi)(A_0-\eta),
\quad
4ts=2t(\xi-\eta).
\]
Moreover,
\begin{equation}\label{QQ-39'}
B_0+\xi\lesssim\xi,
\quad
A_0-\eta\lesssim t,
\end{equation}
and hence
\begin{equation}\label{QQ-39}
Q(t,s,r-q)
=\f{(B_0+\xi)(A_0-\eta)}
{2t(\xi-\eta)}
\lesssim
\f{\xi}{\xi-\eta}.
\end{equation}
For $r-q\geq0$, using $|r-q|\leq A_0-\eta$, one has
\[
\f{|r-q|}{D(t,s,r-q)}
\leq
\f{A_0-\eta}{(B_0+\xi)(A_0-\eta)}
\leq
\xi^{-1},
\]
while for $r-q<0$, using $|r-q|\leq B_0+\xi$ and
$2q\leq\xi\leq A_0$, one can obtain
\begin{equation}\label{QQ-40}
\f{|r-q|}{D(t,s,r-q)}
\leq
(A_0-\eta)^{-1}
\leq
\xi^{-1}.
\end{equation}
This, together with \eqref{QQ-20} and
\eqref{QQ-39}--\eqref{QQ-40}, gives
\[
|\Ecal_\lambda(t,s,r-q)|
\lesssim
Q(t,s,r-q)^\lambda
\lesssim
\xi^\lambda(\xi-\eta)^{-\lambda},
\]
and
\[
|\p_y\Ecal_\lambda(t,s,r-q)|
\lesssim
Q(t,s,r-q)^\lambda
\f{|r-q|}{D(t,s,r-q)}
\lesssim
\xi^{\lambda-1}(\xi-\eta)^{-\lambda}.
\]
Thus \eqref{qq:45} is shown.

For $1\leq s\leq t-r$ and
$0\leq q\leq\min\{s,t-s-r\}$, one has $1\leq\xi\leq B_0$.
By the evenness of $\Ecal_\lambda(t,s,\cdot)$, we get
\begin{equation}\label{QQ-42}
\Ecal_\lambda(t,s,r-q)-\Ecal_\lambda(t,s,r+q)
=
-\int_{|r-q|}^{r+q}
\p_y\Ecal_\lambda(t,s,y)\dd y.
\end{equation}
It follows from \eqref{QQ-20} and \eqref{QQ-39'} that for $|r-q|\leq y\leq r+q$ and $1\leq\xi\leq B_0$,
\[
D(t,s,y)
\leq
D(t,s,r-q)
=
(B_0+\xi)(A_0-\eta)
\lesssim
tB_0
\]
and
\[
D(t,s,y)
\geq
D(t,s,r+q)
=
(A_0+\xi)(B_0-\eta)\geq A_0(B_0-\eta).
\]
Then
\[
Q(t,s,y)=
\f{D(t,s,y)}{2t(\xi-\eta)}
\lesssim
\f{B_0}{\xi-\eta}.
\]
Therefore, for $|r-q|\leq y\leq r+q$,
\begin{equation}\label{QQ-41}
|\p_y\Ecal_\lambda(t,s,y)|\lesssim
Q(t,s,y)^\lambda
\f{|y|}{D(t,s,y)}
\lesssim
\f{(r+q)B_0^\lambda}
{A_0(B_0-\eta)}
(\xi-\eta)^{-\lambda}
.
\end{equation}
Since $\p_y\Ecal_\lambda(t,s,\cdot)$ is odd, using
\eqref{QQ-41} at $y=|r-q|$ and $y=r+q$, one has
\begin{equation}\label{QQ-43}
|\p_y\Ecal_\lambda(t,s,r-q)|
+
|\p_y\Ecal_\lambda(t,s,r+q)|
\lesssim
\f{(r+q)B_0^\lambda}
{A_0(B_0-\eta)}
(\xi-\eta)^{-\lambda}.
\end{equation}
It follows from \eqref{QQ-42} and \eqref{QQ-41} that
\[
\begin{aligned}
\left|
\Ecal_\lambda(t,s,r-q)
-
\Ecal_\lambda(t,s,r+q)
\right|&\lesssim
\f{(r+q)B_0^\lambda}
{A_0(B_0-\eta)}
(\xi-\eta)^{-\lambda}
\left(r+q-|r-q|\right)\\
&=\f{2(r+q)\min\{r,q\}B_0^\lambda}
{A_0(B_0-\eta)}
(\xi-\eta)^{-\lambda}\\
&\lesssim
\f{rqB_0^\lambda}
{A_0(B_0-\eta)}
(\xi-\eta)^{-\lambda}.
\end{aligned}
\]
This, together with \eqref{QQ-43}, derives \eqref{qq:46}.
\end{proof}

To apply the kernel estimates in Lemma~\ref{lem:characteristic-kernel},
we also need the following integral estimate.

\begin{lemma}
\label{lem:characteristic-integral}
Let $\kappa$ satisfy \eqref{eq:kappa-range}. For $j=1,2$ and
$\xi\geq1$, define
\[
I_j(\xi)
=
\int_{-\xi}^{0}
(\xi+\eta)^j
(\xi-\eta)^{-\beta-\lambda}
(1-\eta)^{-p(\kappa-1)}\dd\eta.
\]
Then
\begin{equation}
\label{eq:Ij}
I_j(\xi)
\lesssim
(1+\xi)^{p+j-1-\kappa-\lambda},
\qquad j=1,2.
\end{equation}
\end{lemma}

\begin{proof}
Set $m=p(\kappa-1)$. By the change of variable $\theta=-\eta$ and
$\xi\leq\xi+\theta\leq2\xi$ for $0\leq\theta\leq\xi$, one can obtain
\[
I_j(\xi)
\lesssim
\xi^{-\beta-\lambda}
\int_0^\xi
(\xi-\theta)^j(1+\theta)^{-m}\dd\theta.
\]
A direct computation yields
\[
\begin{aligned}
\int_0^\xi
(\xi-\theta)^j(1+\theta)^{-m}\dd\theta
&\lesssim
\xi^j
\int_0^{\xi/2}
(1+\theta)^{-m}\dd\theta
+
\xi^{-m}
\int_{\xi/2}^{\xi}
(\xi-\theta)^j\dd\theta\\
&=
\xi^j
\int_0^{\xi/2}
(1+\theta)^{-m}\dd\theta
+
\xi^{-m}
\int_0^{\xi/2}
y^j\dd y\\
&\lesssim
\begin{cases}
\xi^{j+1-m},&0<m<1,\\
\xi^j\log(2+\xi),&m=1,\\
\xi^j,&m>1.
\end{cases}
\end{aligned}
\]
Consequently,
\[
I_j(\xi)
\lesssim
\begin{cases}
\xi^{j+1-m-\beta-\lambda},&0<m<1,\\
\xi^{j-\beta-\lambda}\log(2+\xi),&m=1,\\
\xi^{j-\beta-\lambda},&m>1.
\end{cases}
\]
For $0<m<1$ and $\xi\geq1$, one has
\[
\xi^{j+1-m-\beta-\lambda}
\lesssim
(1+\xi)^{p+j-1-\kappa-\lambda},
\]
where we have used the fact of $(p-1)\kappa>2-\beta$. For $m=1$ and
$\xi\geq1$, one then has
\begin{equation}\label{DELTA}
\xi^{j-\beta-\lambda}\log(2+\xi)
\lesssim
\xi^{j-\beta-\lambda}(1+\xi)^\delta
\lesssim
(1+\xi)^{p+j-1-\kappa-\lambda},
\end{equation}
where $\delta=p+\beta-1-\kappa>0$. For $m>1$ and $\xi\geq1$, due to
$\kappa<p+\beta-1$, we then have
\[
\xi^{j-\beta-\lambda}
\lesssim
(1+\xi)^{p+j-1-\kappa-\lambda}.
\]
Therefore, \eqref{eq:Ij} holds for $j=1,2$.
\end{proof}

Based on Lemmas \ref{lem:nonlinear-representation}-\ref{lem:characteristic-integral}, we now prove the following weighted estimates.

\begin{proposition}
\label{prop:nonlinear}
Let $\kappa$ satisfy \eqref{eq:kappa-range}. Then, for any
$w\in X_\kappa$, $t\geq1$, and $0\leq r\leq t$,
\begin{equation}
\label{qq:47}
|\Tcal w(t,r)|
+
\f{1}{1+r}
\left|
\p_r(r\Tcal w)(t,r)
\right|
\lesssim
\|w\|_{X_\kappa}^{p}
\Acal(t,r)^{-1}\Bcal(t,r)^{1-\kappa}.
\end{equation}
In particular,
\begin{equation}
\label{qq:48}
\|\Tcal w\|_{X_\kappa}
\lesssim
\|w\|_{X_\kappa}^{p}.
\end{equation}
Moreover, for any $w,v\in X_\kappa$,
\begin{equation}
\label{qq:49}
\|\Tcal w-\Tcal v\|_{X_\kappa}
\lesssim
\big(
\|w\|_{X_\kappa}^{p-1}
+
\|v\|_{X_\kappa}^{p-1}
\big)
\|w-v\|_{X_\kappa}.
\end{equation}
\end{proposition}

\begin{proof}
Set
$
M=\|w\|_{X_\kappa}.
$
The support condition in $X_\kappa$ gives $G_w(s,q)=0$ for $q>s$.
For $0\leq q\leq s$, it follows from the definition of $X_\kappa$ in \eqref{eq:Xkappa}
that
\begin{equation}
\label{qq:50'}
|w(s,q)|
\leq
M(1+s+q)^{-1}(1+s-q)^{1-\kappa},
\end{equation}
and then
\begin{equation}
\label{qq:50}
|G_w(s,q)|
\lesssim
M^p q s^{-\beta}
(1+s+q)^{-p}
(1+s-q)^{-p(\kappa-1)}.
\end{equation}

We first estimate $r(\Tcal w)(t,r)$ for $r>0$. To this end, we consider the first
term on the right-hand side of \eqref{qq:40}
\[
\int_1^t
\int_{|t-s-r|}^{t-s+r}
\Ecal_\lambda(t,s,r-q)
G_w(s,q)\dd q\dd s.
\]
Under the change of variables in \eqref{qq:44}, the condition
$0\leq q\leq s$ implies
$
-\xi\leq\eta\leq0.
$
Moreover, recalling that $B_0=t-r$ and $A_0=t+r$ in \eqref{qq:44},
the condition $|t-s-r|\leq q\leq t-s+r$ gives
\[
\xi=s+q
\geq
s+|B_0-s|
\geq
B_0,
\quad
\xi\geq1,
\quad
\xi=s+q\leq A_0.
\]
Then
\[
\max\{1,B_0\}\leq\xi\leq A_0.
\]
Therefore, it follows from \eqref{qq:44}-\eqref{qq:45}, \eqref{qq:50} and \eqref{eq:Ij} that
\begin{equation}\label{QQ-44}
\begin{aligned}
\big|
\int_1^t
\int_{|t-s-r|}^{t-s+r}
\Ecal_\lambda(t,s,r-q)
G_w(s,q)\dd q\dd s
\big|
&\lesssim
M^p
\int_{\max\{1,B_0\}}^{A_0}
\xi^\lambda(1+\xi)^{-p}
I_1(\xi)\dd\xi\\
&\lesssim
M^p
\int_{B_0}^{A_0}
(1+\xi)^{-\kappa}\dd\xi\\
&=
\f{M^p}{\kappa-1}
\big[
\Bcal(t,r)^{1-\kappa}
-
\Acal(t,r)^{1-\kappa}
\big]\\
&=\f{M^p}{\kappa-1}\Bcal(t,r)^{1-\kappa}
\big[
1-
\big(
\f{\Bcal(t,r)}{\Acal(t,r)}
\big)^{\kappa-1}
\big]\\
&\lesssim
M^p r
\Acal(t,r)^{-1}
\Bcal(t,r)^{1-\kappa},
\end{aligned}
\end{equation}
where the last inequality follows from
$1-z^{\kappa-1}\lesssim1-z$ for $0\leq z\leq1$ and
$\Acal(t,r)-\Bcal(t,r)=2r$.

We next treat the remaining term on the right-hand side of
\eqref{qq:40}
\begin{equation}\label{YHC-020}
\begin{aligned}
\int_1^{t-r}
\int_0^{t-s-r}
\left[
\Ecal_\lambda(t,s,r-q)
-
\Ecal_\lambda(t,s,r+q)
\right]
G_w(s,q)\dd q\dd s.
\end{aligned}
\end{equation}
Note that this term vanishes when $B_0=t-r\leq1$. For $B_0>1$, the
integration region in \eqref{YHC-020} is
\[
1\leq s\leq B_0,
\quad
0\leq q\leq\min\{s,B_0-s\},
\]
which  imply
\[
1\leq\xi\leq B_0,
\quad
-\xi\leq\eta\leq0,\quad B_0-\eta\geq B_0.
\]
Then it follows from \eqref{qq:46},
\eqref{qq:50} and \eqref{eq:Ij} that
\begin{equation}\label{QQ-45}
\begin{aligned}
&\big|
\int_1^{t-r}
\int_0^{t-s-r}
\left[
\Ecal_\lambda(t,s,r-q)
-
\Ecal_\lambda(t,s,r+q)
\right]
G_w(s,q)\dd q\dd s
\big|\\
&\lesssim
M^p rA_0^{-1}B_0^{\lambda-1}
\int_1^{B_0}
(1+\xi)^{-p}I_2(\xi)\dd\xi\\
&\lesssim
M^p rA_0^{-1}B_0^{\lambda-1}
\int_1^{B_0}
(1+\xi)^{1-\kappa-\lambda}\dd\xi\\
&\lesssim
M^p rA_0^{-1}B_0^{1-\kappa}\\
&\lesssim
M^p r
\Acal(t,r)^{-1}
\Bcal(t,r)^{1-\kappa},
\end{aligned}
\end{equation}
where the third and last inequalities follow from
$\kappa+\lambda<2$ together with
$A_0\sim\Acal(t,r)$ and $B_0\sim\Bcal(t,r)$ for $B_0>1$. Combining \eqref{QQ-44} with
\eqref{QQ-45} yields
\begin{equation}
\label{qq:51}
|r(\Tcal w)(t,r)|
\lesssim
M^p r
\Acal(t,r)^{-1}
\Bcal(t,r)^{1-\kappa}.
\end{equation}

We next turn to deal with $\p_r(r\Tcal w)(t,r)$. From \eqref{qq:42}, we now estimate the first term of $\p_r(r\Tcal w)(t,r)$
\begin{equation}\label{YHC-08}
\begin{aligned}
4^\lambda
\int_1^t
\left[
G_w(s,t-s+r)
+
G_w(s,t-s-r)
\right]\dd s.
\end{aligned}
\end{equation}
For $G_w(s,t-s+r)$, the support condition
$t-s+r=A_0-s\leq s$ implies $s\geq\f{A_0}{2}$.
Setting $\theta=2s-A_0$, one has
\begin{equation}\label{QQ-46}
0\leq\theta\leq B_0,
\quad
1+s+(t-s+r)=\Acal(t,r),
\quad
1+s-(t-s+r)=1+\theta.
\end{equation}
Moreover, $q=t-s+r\leq A_0$ and $s\geq \f{A_0}{2}$ imply
\begin{equation}\label{QQ-46'}
q s^{-\beta}
\lesssim
A_0^{1-\beta}
\lesssim
\Acal(t,r)^{1-\beta}.
\end{equation}
Combining \eqref{QQ-45},  \eqref{QQ-46'} and \eqref{qq:50} yields
\begin{equation}\label{QQ-49}
\begin{aligned}
\int_1^t
|G_w(s,t-s+r)|\dd s
&\lesssim
M^p
\Acal(t,r)^{1-p-\beta}
\int_0^{B_0}
(1+\theta)^{-p(\kappa-1)}\dd\theta\lesssim
M^p\Acal(t,r)^{-\kappa}.
\end{aligned}
\end{equation}
Indeed, if $p(\kappa-1)<1$, then
\[
\int_0^{B_0}
(1+\theta)^{-p(\kappa-1)}\dd\theta
\lesssim
\Bcal(t,r)^{1-p(\kappa-1)}
\leq
\Acal(t,r)^{1-p(\kappa-1)}.
\]
Hence
\begin{equation}\label{YHC-09}
\begin{aligned}
M^p\Acal(t,r)^{1-p-\beta}
\int_0^{B_0}
(1+\theta)^{-p(\kappa-1)}\dd\theta
\lesssim M^p
\Acal(t,r)^{2-p-\beta-p(\kappa-1)}
\lesssim
M^p\Acal(t,r)^{-\kappa},
\end{aligned}
\end{equation}
where  we have used the fact of
$(p-1)\kappa>2-\beta$ in the last inequality.
If $p(\kappa-1)>1$, then the integral $\int_0^{B_0}
(1+\theta)^{-p(\kappa-1)}\dd\theta$ is uniformly bounded, and thus the condition
$\kappa<p+\beta-1$ implies \eqref{QQ-49}.
If $p(\kappa-1)=1$, as treated in \eqref{DELTA}, then \eqref{YHC-09} still holds.

We next treat the term which comes from \eqref{YHC-08}
\begin{equation}\label{YHC-011}
\begin{aligned}
\int_1^t
|G_w(s,t-s-r)|\dd s
=
\int_1^t
|G_w(s,B_0-s)|\dd s.
\end{aligned}
\end{equation}
First, it is assumed that $B_0>1$. Then
\begin{equation}\label{YHC-010}
\begin{aligned}
\int_1^t
|G_w(s,B_0-s)|\dd s
=
\int_1^{B_0}
|G_w(s,B_0-s)|\dd s
+
\int_{B_0}^{t}
|G_w(s,B_0-s)|\dd s .
\end{aligned}
\end{equation}
The first integral on the right-hand side of \eqref{YHC-010} follows from the estimate for
$G_w(s,t-s+r)$ by replacing $A_0$ with $B_0$. Thus,
\begin{equation}
\label{qq:54}
\int_1^{B_0}
|G_w(s,B_0-s)|\dd s
\lesssim
M^p\Bcal(t,r)^{-\kappa}.
\end{equation}
For the second integral on the right-hand side of \eqref{YHC-010}, by $s\geq B_0$ and the oddness of
$G_w(s,\cdot)$, one has
\[
|G_w(s,B_0-s)|
=|G_w(s,s-B_0)|.
\]
Using \eqref{qq:50} with
$q=s-B_0$ and $\xi=s+q=2s-B_0$, we arrive at
\begin{equation}
\label{qq:55}
\begin{aligned}
&\int_{B_0}^{t}
|G_w(s,B_0-s)|\dd s\lesssim
M^p(1+B_0)^{-p(\kappa-1)}
\int_{B_0}^{A_0}
(\xi+B_0)^{-\beta}
(1+\xi)^{-p}
(\xi-B_0)\dd\xi.
\end{aligned}
\end{equation}
In addition, by $B_0>1$ and $p+\beta>2$, it holds that
\begin{equation}
\label{qq:56}
\begin{aligned}
&\int_{B_0}^{A_0}
(\xi+B_0)^{-\beta}
(1+\xi)^{-p}
(\xi-B_0)\dd\xi\lesssim
\int_{B_0}^{2B_0}
B_0^{-\beta-p}
(\xi-B_0)\dd\xi
+\int_{2B_0}^{\infty}
\xi^{1-p-\beta}\dd\xi\lesssim
B_0^{2-p-\beta}.
\end{aligned}
\end{equation}
Then it follows from \eqref{qq:55} and the condition $(p-1)\kappa>2-\beta$ that
\begin{equation}
\label{qq:57}
\begin{aligned}
\int_{B_0}^{t}
|G_w(s,B_0-s)|\dd s
&\lesssim
M^p
\Bcal(t,r)^{-p(\kappa-1)}
B_0^{2-p-\beta}\lesssim
M^p\Bcal(t,r)^{-\kappa}.
\end{aligned}
\end{equation}
Combining \eqref{qq:54} and \eqref{qq:57} yields that for $B_0>1$,
\begin{equation}
\label{qq:58}
\int_1^t
|G_w(s,t-s-r)|\dd s
\lesssim
M^p\Bcal(t,r)^{-\kappa}.
\end{equation}
Next, consider the case of $B_0\leq1$ in \eqref{YHC-011}. Due to
$\Bcal(t,r)=1+B_0\sim1$,
we can arrive at $|B_0-s|=s-B_0$ for $s\geq1$. Substituting
$q=s-B_0$ into \eqref{qq:50} yields
\[
\begin{aligned}
|G_w(s,B_0-s)|
&\lesssim
M^p(s-B_0)s^{-\beta}
(1+2s-B_0)^{-p}
(1+B_0)^{-p(\kappa-1)}\lesssim
M^p(1+s)^{1-p-\beta}.
\end{aligned}
\]
Then the conditions $p+\beta>2$ and $\Bcal(t,r)\sim1$ imply
\[
\int_1^t
|G_w(s,B_0-s)|\dd s
\lesssim
M^p\int_1^\infty
(1+s)^{1-p-\beta}\dd s
\lesssim
M^p,
\]
which means
\begin{equation}\label{QQ-48}
\int_1^t
|G_w(s,t-s-r)|\dd s
\lesssim
M^p\Bcal(t,r)^{-\kappa}.
\end{equation}
Therefore, collecting the results in \eqref{QQ-49}, \eqref{qq:58} and  \eqref{QQ-48}, we conclude
\begin{equation}\label{QQ-50}
\begin{aligned}
4^\lambda \big|
\int_1^t
\left[
G_w(s,t-s+r)
+G_w(s,t-s-r)
\right]\dd s \big| 
&\lesssim
M^p
\left(
\Acal(t,r)^{-\kappa}
+\Bcal(t,r)^{-\kappa}
\right)\\
&\lesssim
M^p(1+r)
\Acal(t,r)^{-1}
\Bcal(t,r)^{1-\kappa},
\end{aligned}
\end{equation}
where in the last inequality we have used the facts of
$\Acal(t,r)=\Bcal(t,r)+2r$ and $\Bcal(t,r)\geq1$.

We next consider the second term on the right-hand side of
\eqref{qq:42}, namely,
\begin{equation}\label{YHC-013}
\begin{aligned}
\int_1^t
\int_{|t-s-r|}^{t-s+r}
\partial_y\Ecal_\lambda(t,s,r-q)
G_w(s,q)\dd q\dd s .
\end{aligned}
\end{equation}
Using the change of variables in \eqref{qq:44},  from the integration region in \eqref{YHC-013}, one has
\[
\xi=s+q\leq A_0,
\quad
\xi=s+q\geq s+|B_0-s|\geq B_0,
\]
while $\xi=s+q\geq1$. Hence,
\[
\max\{1,B_0\}\leq\xi\leq A_0.
\]
It follows from \eqref{qq:45} and \eqref{qq:50} that
\begin{equation}
\label{qq:62}
\begin{aligned}
&
\big|
\int_1^t
\int_{|t-s-r|}^{t-s+r}
\partial_y\Ecal_\lambda(t,s,r-q)
G_w(s,q)\dd q\dd s
\big|
\lesssim
M^p
\int_{\max\{1,B_0\}}^{A_0}
\xi^{\lambda-1}
(1+\xi)^{-p}
I_1(\xi)
\dd\xi .
\end{aligned}
\end{equation}
Applying \eqref{eq:Ij} with $j=1$ yields
\[
I_1(\xi)
\lesssim
(1+\xi)^{p-\kappa-\lambda}.
\]
Therefore,
\begin{equation}
\label{qq:63}
\begin{aligned}
&
\big|
\int_1^t
\int_{|t-s-r|}^{t-s+r}
\partial_y\Ecal_\lambda(t,s,r-q)
G_w(s,q)\dd q\dd s
\big|
\lesssim
M^p
\int_{B_0}^{A_0}
(1+\xi)^{-\kappa-1}
\dd\xi .
\end{aligned}
\end{equation}
Similarly treated as in \eqref{QQ-44}, one has
\begin{equation}
\label{qq:64}
\begin{aligned}
&
\big|
\int_1^t
\int_{|t-s-r|}^{t-s+r}
\partial_y\Ecal_\lambda(t,s,r-q)
G_w(s,q)\dd q\dd s
\big|
\lesssim
M^p
r\Acal(t,r)^{-1}
\Bcal(t,r)^{1-\kappa}.
\end{aligned}
\end{equation}

It remains to consider the last term on the right-hand side of
\eqref{qq:42}
\[
\int_1^{t-r}
\int_0^{t-s-r}
\left[
\partial_y\Ecal_\lambda(t,s,r-q)
-
\partial_y\Ecal_\lambda(t,s,r+q)
\right]
G_w(s,q)\dd q\dd s,
\]
which vanishes when $B_0\leq1$. For $B_0>1$, by the same
change of variables as in the proof of \eqref{QQ-45}, together
with \eqref{qq:46}, \eqref{qq:50} and \eqref{eq:Ij}, we have
\begin{equation}
\label{qq:64'}
\begin{aligned}
&
\big|
\int_1^{t-r}
\int_0^{t-s-r}
\left[
\partial_y\Ecal_\lambda(t,s,r-q)
-
\partial_y\Ecal_\lambda(t,s,r+q)
\right]
G_w(s,q)\dd q\dd s
\big|\\
&\lesssim
M^p
A_0^{-1}B_0^{\lambda-1}
\big[
r
\int_1^{B_0}
(1+\xi)^{-p}
I_1(\xi)\dd\xi
+
\int_1^{B_0}
(1+\xi)^{-p}
I_2(\xi)\dd\xi
\big].
\end{aligned}
\end{equation}
By \eqref{eq:Ij}, one can obtain
\[
\begin{aligned}
&
r
\int_1^{B_0}
(1+\xi)^{-p}
I_1(\xi)\dd\xi
+
\int_1^{B_0}
(1+\xi)^{-p}
I_2(\xi)\dd\xi\lesssim
r
\int_1^{B_0}
(1+\xi)^{-\kappa-\lambda}
\dd\xi
+
\int_1^{B_0}
(1+\xi)^{1-\kappa-\lambda}
\dd\xi .
\end{aligned}
\]
Since
$
1<\kappa+\lambda<2,
$
it follows that
\[
\begin{aligned}
&
r
\int_1^{B_0}
(1+\xi)^{-\kappa-\lambda}
\dd\xi
+
\int_1^{B_0}
(1+\xi)^{1-\kappa-\lambda}
\dd\xi\lesssim
r+B_0^{2-\kappa-\lambda}.
\end{aligned}
\]
Then
\begin{equation}
\label{qq:65}
\begin{aligned}
&
\big|
\int_1^{t-r}
\int_0^{t-s-r}
\left[
\partial_y\Ecal_\lambda(t,s,r-q)
-
\partial_y\Ecal_\lambda(t,s,r+q)
\right]
G_w(s,q)\dd q\dd s
\big|\\
&\lesssim
M^pA_0^{-1}B_0^{\lambda-1}
\big(
r+B_0^{2-\kappa-\lambda}
\big)\\
&\lesssim
M^pA_0^{-1}
\big(
rB_0^{\lambda-1}
+B_0^{1-\kappa}
\big).
\end{aligned}
\end{equation}
Due to $B_0>1$, one has
$A_0\sim\Acal(t,r)$ and $B_0\sim\Bcal(t,r)$.
Then it follows from the condition $\kappa<2-\lambda$  that
\begin{equation}
\label{qq:66}
\begin{aligned}
&
\big|
\int_1^{t-r}
\int_0^{t-s-r}
\left[
\partial_y\Ecal_\lambda(t,s,r-q)
-\partial_y\Ecal_\lambda(t,s,r+q)
\right]
G_w(s,q)\dd q\dd s
\big|\\
&\lesssim
M^p\Acal(t,r)^{-1}
\left(
r\Bcal(t,r)^{1-\kappa}
+\Bcal(t,r)^{1-\kappa}
\right)\\
&\lesssim
M^p(1+r)
\Acal(t,r)^{-1}
\Bcal(t,r)^{1-\kappa}.
\end{aligned}
\end{equation}
By \eqref{QQ-50}, \eqref{qq:64} and
\eqref{qq:66}, we can conclude
\begin{equation}
\label{qq:67}
\left|
\partial_r(r\Tcal w)(t,r)
\right|
\lesssim
M^p(1+r)
\Acal(t,r)^{-1}
\Bcal(t,r)^{1-\kappa}.
\end{equation}
Combining \eqref{qq:51} and \eqref{qq:67} yields that for $r>0$,
\begin{equation}
\label{qq:68}
|\Tcal w(t,r)|
+\frac1{1+r}
\left|
\partial_r(r\Tcal w)(t,r)
\right|
\lesssim
M^p
\Acal(t,r)^{-1}
\Bcal(t,r)^{1-\kappa}.
\end{equation}
For $r=0$, the oddness of $r\Tcal w(t,r)$ with respect to $r$
implies
\[
\Tcal w(t,0)
=\lim_{r\to0}
\frac{r\Tcal w(t,r)}{r}
=\partial_r(r\Tcal w)(t,0).
\]
The estimate at $r=0$ follows from \eqref{qq:67} by letting
$r$ approach zero.
Hence, \eqref{qq:47} holds for $0\leq r\leq t$.
Then \eqref{qq:48} follows immediately from the definition of
$X_\kappa$ and \eqref{qq:47}.

It remains to prove \eqref{qq:49}. Notice that for any $w,v\in X_\kappa$, one has
\[
\bigl||w|^p-|v|^p\bigr|
\lesssim
|w-v|
\left(
|w|^{p-1}
+|v|^{p-1}
\right).
\]
Applying \eqref{qq:50'} to $w$, $v$, and $w-v$, respectively,
one obtains
\[
\begin{aligned}
|G_w(s,q)-G_v(s,q)|
&\lesssim
q s^{-\beta}
\bigl||w(s,q)|^p-|v(s,q)|^p\bigr|
\\
&\lesssim
\big(\|w\|_{X_\kappa}^{p-1}
+\|v\|_{X_\kappa}^{p-1}
\big)
\|w-v\|_{X_\kappa}
q s^{-\beta}
(1+s+q)^{-p}
(1+s-q)^{-p(\kappa-1)}.
\end{aligned}
\]
Since the above expression has the same form as \eqref{qq:50},
the estimates derived above remain valid with $G_w$ replaced by
$G_w-G_v$.
Therefore, \eqref{qq:49}  follows immediately from the resulting
estimate and the definition of $X_\kappa$. This completes the proof.
\end{proof}
\section{Proof of Theorem~\ref{thm:main}}\label{Sect5}

Based on Proposition~\ref{prop:linear} and
Proposition~\ref{prop:nonlinear}, we now prove Theorem~\ref{thm:main}.
\vskip 0.1 true cm

\begin{proof}
[Proof of Theorem~\ref{thm:main}]
Define the mapping
\[
\Phi (w)=w_{\rm lin}+\Tcal w,
\]
where $w_{\rm lin}$ denotes the solution of the linear homogeneous
problem considered in Section~\ref{sec:linear-estimates}, and $\Tcal$ is the nonlinear
operator defined in Section~\ref{sec:nonlinear-Riemann-operator}.
For a fixed constant $M>0$, we introduce the closed subset
\[
X_\kappa(M\varepsilon)
=
\{w\in X_\kappa:\|w\|_{X_\kappa}\leq M\varepsilon\}.
\]
Choosing
$M=3C$, we now show that $\Phi$ is a contraction mapping on
$X_\kappa(M\varepsilon)$.

Indeed, for any $w\in X_\kappa(M\varepsilon)$, Propositions
\ref{prop:linear} and \ref{prop:nonlinear} imply that
\[
\begin{aligned}
\|\Phi(w)\|_{X_\kappa}
&\leq
\|w_{\rm lin}\|_{X_\kappa}
+
\|\Tcal w\|_{X_\kappa}
\lesssim
\varepsilon+\|w\|_{X_\kappa}^{p}
\leq
C\varepsilon+C(M\varepsilon)^p
=
\left(
C+CM^p\varepsilon^{p-1}
\right)\varepsilon.
\end{aligned}
\]
Choose $\varepsilon_1>0$ sufficiently small such that
\[
CM^p\varepsilon_1^{p-1}\leq2C .
\]
Then, for any $0<\varepsilon\leq\varepsilon_1$, one has
\[
\|\Phi(w)\|_{X_\kappa}
\leq
3C\varepsilon
=
M\varepsilon .
\]
Thus, $\Phi$ maps $X_\kappa(M\varepsilon)$ into itself.

Moreover, for any $w,v\in X_\kappa(M\varepsilon)$, it follows from
Proposition~\ref{prop:nonlinear} that
\[
\begin{aligned}
\|\Phi(w)-\Phi(v)\|_{X_\kappa}
&=
\|\Tcal w-\Tcal v\|_{X_\kappa}
\lesssim
\big(\|w\|_{X_\kappa}^{p-1}
+\|v\|_{X_\kappa}^{p-1}
\big)
\|w-v\|_{X_\kappa}
\leq
2C(M\varepsilon)^{p-1}
\|w-v\|_{X_\kappa}.
\end{aligned}
\]
Choose $\varepsilon_2>0$ sufficiently small such that
\[
2C(M\varepsilon_2)^{p-1}\leq\frac12 .
\]
Then, for any $0<\varepsilon\leq\varepsilon_2$, one can obtain
\[
\|\Phi(w)-\Phi(v)\|_{X_\kappa}
\leq
\frac12
\|w-v\|_{X_\kappa}.
\]
Taking
$
\varepsilon_0=\min\{\varepsilon_1,\varepsilon_2\},
$
we can conclude that, for any
$0<\varepsilon\leq\varepsilon_0$, $\Phi$ is a contraction mapping on
$X_\kappa(M\varepsilon)$.
Therefore, there exists a unique fixed point
$w\in X_\kappa(M\varepsilon)$ satisfying
$
\Phi w=w.
$
The corresponding function
\[
u(t,r)=t^{-\frac{\mu}{2}}w(t,r)
\]
is a unique global radial solution of \eqref{qq:1}.
By $w\in X_\kappa$, both $w$ and $\p_r(rw)$ are continuous. For
$r\neq0$, the identity
$w_r=r^{-1}\bigl(\p_r(rw)-w\bigr)$
shows that $w_r$ is continuous. Together with the continuity of $w$
and the condition $p>1$, one obtains the continuity of $G_w$ and
$\p_rG_w$ for $r\neq0$.
Hence, the integral formula defining $\Phi$ is differentiated in
$t$ and $r$ for $r\neq0$, which yields
$\psi=rw\in C^2([1,\infty)\times(\R\setminus\{0\}))$.
Therefore,
$u(t,x)=t^{-\mu/2}w(t,|x|)$ satisfies
\[
u\in C\bigl([1,\infty)\times\Bbb R^3\bigr)
\cap C^2\bigl([1,\infty)\times(\Bbb R^3\setminus\{0\})\bigr).
\]
This completes the proof of Theorem~\ref{thm:main}.
\end{proof}

\vskip 0.1 true cm

{\bf Data availability statement}. All data that support the findings of this study
are included within the article (and any supplementary files).


\end{document}